\documentclass[10pt,a4paper]{amsart}
\usepackage[latin9]{inputenc}
\usepackage{amsthm}
\usepackage{amstext}
\usepackage{amssymb}

\usepackage[unicode=true,pdfusetitle,
 bookmarks=true,bookmarksnumbered=false,bookmarksopen=false,
 breaklinks=false,pdfborder={0 0 1},backref=false,colorlinks=false]
 {hyperref}

\makeatletter

\numberwithin{equation}{section}
\numberwithin{figure}{section}
\theoremstyle{plain}
\newtheorem{thm}{\protect\theoremname}
  \theoremstyle{definition}
  \newtheorem{defn}[thm]{\protect\definitionname}
  \theoremstyle{remark}
  \newtheorem{rem}[thm]{\protect\remarkname}
  \theoremstyle{plain}
  \newtheorem{prop}[thm]{\protect\propositionname}
  \theoremstyle{definition}
  \newtheorem{example}[thm]{\protect\examplename}

\usepackage{amssymb,amsthm,amsmath,amsfonts,amscd}
\usepackage{amsaddr}
\usepackage{graphicx}
\usepackage{url}
\usepackage{color}
\usepackage[active]{srcltx}
\usepackage[matrix,arrow]{xy}
\usepackage{mathrsfs}
\usepackage{enumerate}
\usepackage{amsopn} 
\usepackage{bbm} 

\usepackage{cite}
\usepackage{hyperref}
\allowdisplaybreaks[1]
\usepackage[english]{babel}
\usepackage{bbm}
\usepackage{mhchem}
\usepackage{mdef}
\date{\today}


\usepackage{stmaryrd}

\makeatother

  \providecommand{\definitionname}{Definition}
  \providecommand{\examplename}{Example}
  \providecommand{\propositionname}{Proposition}
  \providecommand{\remarkname}{Remark}
\providecommand{\theoremname}{Theorem}

\begin{document}

\title[Stability of solutions to SPDE]{Stability of solutions to stochastic partial differential equations}
\begin{abstract}
\noindent We provide a general framework for the stability of solutions to stochastic partial differential equations with respect to perturbations of the drift. More precisely, we consider stochastic partial differential equations with drift given as the subdifferential of a convex function and prove continuous dependence of the solutions with regard to random Mosco convergence of the convex potentials. In particular, we identify the concept of stochastic variational inequalities (SVI) as a well-suited framework to study such stability properties. The generality of the developed framework is then laid out by deducing Trotter type and homogenization results for stochastic fast diffusion and stochastic singular $p$-Laplace equations. In addition, we provide an SVI treatment for stochastic nonlocal $p$-Laplace equations and prove their convergence to the respective local models.
\end{abstract}

\author[B. Gess]{Benjamin Gess}
\address{Max-Planck Institute for Mathematics in the Sciences \\
Inselstra\ss{}e 22\\
04103 Leipzig\\
Germany}
\email{bgess@mis.mpg.de}

\author[J. M. T\"olle]{Jonas M. Tölle }

\address{Fakultät für Mathematik\\
Universität Bielefeld\\
Postfach 100131\\
33501 Bielefeld\\
Germany}

\email{jonasmtoelle@gmail.com}


\thanks{J.M.T. gratefully acknowledges funding granted by the CRC 701 ``Spectral Structures and Topological Methods in Mathematics'' of the German Research Foundation (DFG)}

\maketitle

\section{Introduction}

We consider the stability of stochastic partial differential equations of the general type
\begin{equation}
dX_{t}\in-\partial\vp(X_{t})dt+B(X_{t})dW_{t}\label{eq:intro_gen_SPDE}
\end{equation}
with respect to perturbations of the convex, lower-semicontinuous potential $\vp$, defined on some separable Hilbert space $H$. Here, $W$ is a cylindrical Wiener process on a separable Hilbert space $U$ and $B:H\to L_{2}(U,H)$ are Lipschitz continuous diffusion coefficients. We are especially interested in applications to quasilinear, singular-degenerate SPDE, such as the stochastic singular $p$-Laplace equation 
\begin{equation}
dX_{t}\in\div\left(|\nabla X_{t}|^{p-2}\nabla X_{t}\right)dt+B(X_{t})dW_{t},\label{eq:intro_splp}
\end{equation}
with $p\in[1,2)$, which will serve as a model example in the introduction. In particular, this generalizes results obtained in \cite{BR13,BDPR09-4,GT11} on the multi-valued case of the stochastic total variation flow ($p=1$).

In the deterministic case, i.e.~$B\equiv0$ in \eqref{eq:intro_gen_SPDE}, the stability of solutions with respect to $\vp$ is well-understood \cite{Att78}. More precisely, for a sequence $\vp^{n}$ of convex, lower-semicontinuous functions on $H$ and corresponding solutions $X^{n}$ it is known that the convergence of $\vp^{n}$ to $\vp$ in Mosco sense (cf.~Appendix \ref{sec:Random-Mosco-convergence} below) implies the convergence of $X^{n}$ to $X$. 

In the stochastic case \eqref{eq:intro_gen_SPDE} much less is known and only particular examples could be treated so far \cite{CT12,C11-1,C11,C09,BBHT13} (cf.~Section \ref{sec:known_res} below). In particular, the singular nature of \eqref{eq:intro_splp} and the resulting low regularity of the solutions lead to difficulties in proving stability with respect to perturbations of the drift $\partial\vp$. In this work we introduce the notion of random Mosco convergence of convex, lower-semicontinuous functionals $\vp^{n}$ and prove that if $\vp^{n}\to\vp$ in random Mosco sense, then the corresponding solutions $X^{n}$ to \eqref{eq:intro_gen_SPDE} converge weakly, that is,
\[
X^{n}\rightharpoonup X\quad\text{in }L^{2}([0,T]\times\O;H).
\]
A key ingredient of the proof of this result is the right choice of a notion of a solution to \eqref{eq:intro_gen_SPDE}. Due to the low regularity of solutions to singular SPDE such as \eqref{eq:intro_splp} (especially for $p=1$), an appropriate notion of a solution needs to rely on little regularity only. We identify the SVI approach to SPDE to be a well-suited framework to study stability questions for SPDE of the type \eqref{eq:intro_gen_SPDE}. 

The abstract convergence results are then applied to a variety of examples, that become immediate consequences of the abstract theory. For the sake of the introduction we shall restrict to the model example of stochastic singular $p$-Laplace equations \eqref{eq:intro_splp}. We provide three classes of applications partially extending results from \cite{C09,C11,C11-1,CT12}:

\textit{Nonlocal approximation:} Consider stochastic singular nonlocal $p$-Laplace equations of the type{\small{
\begin{align}
dX_{t}^{\ve} & \in\left(\int_{\mcO}J^{\ve}\left(\cdot-\xi\right)|X_{t}^{\ve}(\xi)-X_{t}^{\ve}(\cdot)|^{p-2}(X_{t}^{\ve}(\xi)-X_{t}^{\ve}(\cdot))\, d\xi\right)dt+B(X_{t}^{\ve})\, dW_{t}\label{eq:intro:nonlocal_plp}
\end{align}
}}where $p\in[1,2)$,  $J:\R^{d}\to\R$ is a nonnegative, continuous, radial kernel and $J^{\ve}$ is an appropriate rescaling given by
\[
J^{\ve}(z)=\frac{C}{\ve^{p+d}}J\left(\frac{z}{\ve}\right),
\]
with $C$ being some normalization constant. For details see Section \ref{sec:to_local} below. We prove that the solutions $X^{\ve}$ to \eqref{eq:intro:nonlocal_plp} converge to the solution of the stochastic (local) $p$-Laplace equation \eqref{eq:intro_splp}. 

It should be noted that the natural Gelfand triple associated to \eqref{eq:intro:nonlocal_plp} is the trivial triple $V=L^{2}(\mcO)\subseteq H=L^{2}(\mcO)\subseteq V^{*}$, whereas for \eqref{eq:intro_splp} it is $V=(W^{1,p}\cap L^{2})(\mcO)\subseteq H=L^{2}(\mcO)\subseteq V^{*}$. Hence, the approximating solutions $X^{\ve}$ do not satisfy the regularity properties that would be required in order to identify their limit as a variational solution to \eqref{eq:intro_splp}. This lack of regularity makes the proof of convergence to the local model a difficult problem, well beyond existing techniques.

We further note that well-posedness for stochastic quasilinear, non-local SPDE such as \eqref{eq:intro:nonlocal_plp} is proven here for the first time. The developed SVI framework for \eqref{eq:intro:nonlocal_plp} provides a unified framework for all $p\in[1,2)$, in particular including the multi-valued case $p=1$. This joins the two active fields of nonlocal PDE and quasilinear SPDE, giving rise to new, intriguing questions such as convergence to local limits (cf. Section \ref{sec:to_local} below) as well as ergodicity and convergence of invariant measures of nonlocal SPDE, which is addressed in the subsequent work \cite{GT15-2}.

In the deterministic case (i.e.~$B\equiv0$ in \eqref{eq:intro_splp}), nonlocal $p$-Laplace equations have been treated in detail in \cite{AMRT08,AMRT09,AVMRT10,AMRT11} and the references therein. We note that the approach to nonlocal PDE developed in these works is based on the Crandall-Ligget approach to accretive PDE, an approach not applicable in the stochastically perturbed case. We identify the SVI approach to provide an appropriate alternative to prove well-posedness for nonlocal SPDE. 

\textit{Trotter type results: }Consider stochastic generalized $p$-Laplace equations of the type
\begin{equation}
dX_{t}\in\div\phi\left(\nabla X_{t}\right)dt+B(X_{t})dW_{t},\label{eq:intro_trotter}
\end{equation}
where $\phi=\partial\psi$ and $\psi:\R^{d}\to\R_{+}$ is a convex, continuous function with sublinear growth. Assuming $\psi^{n}\to\psi$ in Mosco sense and $\limsup_{n\to\infty}\psi^{n}(z)\le\psi(z)$ for all $z\in\R^{d}$ we prove that the corresponding solutions to \eqref{eq:intro_trotter} converge. In particular, this implies continuous dependence of the solutions to \eqref{eq:intro_splp} on the parameter $p\in[1,2)$. This partially generalizes \cite{C09,C11-1,CT12}.

\textit{Periodic homogenization: }Consider 
\begin{align*}
dX_{t}^{\ve} & =\div\left(a\left(\frac{\xi}{\ve}\right)|\nabla X_{t}^{\ve}|^{p-2}\nabla X_{t}^{\ve}\right)\, dt+B(X_{t}^{\ve})\, dW_{t},
\end{align*}
with $a\in L^{\infty}(\R^{d})$ being periodic, $p\in(1,2)$. We prove that the corresponding solutions $X^{\ve}$ converge to the homogenized limit 
\begin{align*}
dX_{t} & =M_{Y}(a)\div\left(|\nabla X_{t}|^{p-2}\nabla X_{t}\right)\, dt+B(X_{t})\, dW_{t},
\end{align*}
where $M_{Y}(a):=\frac{1}{|Y|}\int_{Y}a(y)\, dy$. This solves the periodic homogenization problem for stochastic \emph{singular} $p$-Laplace equations while previously only degenerate cases, i.e.~$p\ge2$, could be treated. A key difference is the lack of the (compact) embedding of the associated energy space $V=W^{1,p}$ in $L^{2}$ in the singular case $p\in(1,2)$ which renders previous methods inapplicable. This partially generalizes \cite{C11-1,C11}.

\subsection{Overview of known results and comparison\label{sec:known_res}}

In the following we give a brief overview of known stability results for quasilinear SPDE with respect to perturbations of the drift.

In \cite{C09} Trotter type results for stochastic porous media equations with linear multiplicative noise
\begin{equation}
dX_{t}\in\D\psi(X_{t})dt+\sum_{k=1}^{\infty}f_{k}X_{t}d\b_{t}^{k}\label{eq:intro_SPME}
\end{equation}
on bounded, smooth domains $\mcO\subseteq\R^{d}$ with $d\le3$ and $f_{k}\in L^{\infty}(\mcO)$ decaying fast enough have been shown. More precisely, assuming $\psi^{n}\to\psi$ in Mosco sense and appropriate uniform growth conditions, strong convergence of the corresponding solutions $X^{n}$ to $X$ is proven in \cite{C09}. \\
In comparison, Trotter type results to \eqref{eq:intro_SPME} are immediate consequences of our abstract results, without restriction on the dimension $d\in\N.$ Moreover, we treat general diffusion coefficients $B$, thus dispensing with the linearity assumption on the noise in \eqref{eq:intro_SPME}. On the other hand, we only conclude weak convergence of solutions whereas strong convergence was shown in \cite{C09}.

In the subsequent work \cite{C11} these Trotter type results were extended to spatially dependent nonlinearities (again assuming $d\le3$ and linear multiplicative noise), i.e.
\begin{equation}
dX_{t}=\D\psi(\xi,X_{t})dt+\sum_{k=1}^{\infty}f_{k}X_{t}d\b_{t}^{k}\label{eq:intro_SPME-1}
\end{equation}
in order to allow applications to homogenization. In particular, periodic homogenization ($\ve\to0$) of the type
\begin{equation}
dX_{t}^{\ve}=\D\left(a\left(\frac{\xi}{\ve}\right)(X_{t}^{\ve})^{[m]}\right)dt+\sum_{k=1}^{\infty}f_{k}X_{t}^{\ve}d\b_{t}^{k}\label{eq:intro_SPME_homo}
\end{equation}
is shown in \cite{C11} for $m\in[1,5)$ and requiring stringent assumptions on the spatially dependent term $a$. We note that \cite{C11} could only treat the porous medium case ($m>1$) while the fast diffusion case ($m\in(0,1)$) was left as an open problem. An essential difference between these cases is, that in the porous medium case one has the compact embedding of the energy space $V=L^{m+1}$ in $H^{-1}$, while this ceases to be true for $m\in(0,1)$. Again, homogenization for \eqref{eq:intro_SPME_homo} with $m\in(0,1)$ becomes an immediate consequence of our abstract results in general dimension $d$ and for general diffusion coefficients $B$. In addition, our approach allows to relax the assumptions posed on $a$ (cf.~Section \ref{sec:homo_FDE} below). As above, we obtain weak convergence of solutions to \eqref{eq:intro_SPME_homo} to the homogenized SPDE, whereas strong convergence was deduced in \cite{C11} for a smaller class of SPDE.

In \cite{C11-1} a Trotter type theorem for variational SPDE with additive noise
\[
dX_{t}=-\nabla\vp(X_{t})dt+dW_{t}
\]
with respect to perturbations $\vp^{n}\to\vp$ has been shown. For the notion $\nabla\vp$, i.e.~the Gâteaux differential of $\vp$ on $V=\mcD(\vp)$, see \cite{C11-1}. As a crucial assumption, in \cite{C11-1}, the existence of an underlying uniform (in $n$) Gelfand triple $V\subseteq H\subseteq V^{*}$ has been assumed. Roughly speaking, this corresponds to assuming uniform domains for the potentials $\vp^{n}$, that is, $V=\mcD(\vp^{n})$ for all $n\in\N$. While such a condition is satisfied by applications in periodic homogenization, it is not satisfied by Trotter type results as in \eqref{eq:intro_SPME}, neither for nonlocal approximations such as \eqref{eq:intro:nonlocal_plp}. Similarly, in \cite{BBHT13} weak convergence of solutions $X^{n}$ to SPDE of the type
\[
dX_{t}^{n}\in-\partial\vp^{n}(X_{t})dt+\sum_{j=1}^{N}B_{j}^{n}X_{t}^{n}\circ d\b_{t}^{j}
\]
with $B_{j}^{n}$ being linear, commuting operators and $\vp^{n}$ allowing a uniform Gelfand triple was shown. In \cite{C11-1} these abstract results were then used to analyze the periodic homogenization problem for $p$-Laplace equations of the type
\[
dX_{t}^{\ve}=\div\left(a\left(\frac{\xi}{\ve},\nabla X_{t}^{\ve}\right)\right)dt+dW_{t},
\]
assuming, besides several further assumptions, that $a$ is strictly elliptic and strongly monotone with linear growth. In particular, stochastic singular $p$-Laplace equations such as
\begin{equation}
dX_{t}^{\ve}=\div\left(a\left(\frac{\xi}{\ve}\right)|\nabla X_{t}^{\ve}|^{p-2}\nabla X_{t}^{\ve}\right)dt+B(X_{t}^{\ve})dW_{t},\label{eq:intro_plp_homo}
\end{equation}
with $p\in(1,2)$, could not be treated in \cite{C11-1}. In the present work we show that periodic homogenization for \eqref{eq:intro_plp_homo} becomes a direct consequence of our general stability results. This includes general multiplicative noise and singular-degenerate $p$-Laplace drifts, thus partially extending the results from \cite{C11-1}. As before, we deduce weak convergence, while strong convergence was shown in \cite{C11-1}.

The stability of singular $p$-Laplace equations with additive noise
\begin{equation}
dX_{t}\in\div\left(|\nabla X_{t}|^{p-2}\nabla X_{t}\right)dt+dW_{t},\label{eq:intro_plp_homo-1}
\end{equation}
with respect to $p\in[1,2)$ has been investigated in \cite{CT12}, where strong convergence of solutions has been shown, assuming $d\le2$. These results are complemented by the results given in the present paper, by allowing multiplicative noise, removing the dimensional restriction and by providing a general framework for stability of SPDE having stability of \eqref{eq:intro_plp_homo-1} with respect to $p$ as a straightforward consequence. 

For related results in deterministic situations we refer to \cite{LukPar15,KinPar10} and references therein.

Apart from the stability properties for SPDE obtained in this paper, we develop an SVI approach to new classes of quasilinear, singular-degenerate SPDE, such as stochastic nonlocal $p$-Laplace equations. We also prove well-posedness of SVI solutions for the stochastic total variation flow
\begin{equation}
dX_{t}\in\div\left(\sgn(\nabla X_{t})\right)dt+B(X_{t})dW_{t},\label{eq:intro_plp_homo-1-1}
\end{equation}

by means of a different method than used in \cite{BR13}. This significantly simplifies the proof of well-posedness and generalizes the well-posedness results developed in \cite{BR13} by removing dimensional restrictions and by allowing general multiplicative noise, whereas in \cite{BR13} only linear multiplicative noise could be treated.

We would also like to mention the recently developed operatorial approach to SPDE \cite{BarbuRoeckner2015} and the reformulation of SPDE in terms of optimal control problems \cite{Barbu2011,BBHT13}, following the Br\'ezis-Ekeland variational principle, which might prove useful to study stability of SPDE in the future.

\subsection{Structure of the paper}

In Section \ref{sec:generalities_SVI} we introduce the general framework of stochastic variational inequalities and provide the definition of random Mosco convergence. The main result of Section \ref{sec:generalities_SVI} is the proof of convergence of solutions provided random Mosco convergence of the associated potentials holds. In Section \ref{sec:plp} (Section \ref{sec:nonlocal_plp} resp.) well-posedness of SVI solutions to the stochastic (nonlocal resp.) $p$-Laplace equation is shown. Convergence of solutions to the stochastic nonlocal $p$-Laplace equation to the stochastic local $p$-Laplace equation is proven in Section \ref{sec:to_local}. In Section \ref{sec:trotter} Trotter type results are deduced for stochastic $p$-Laplace and stochastic fast diffusion equations. Homogenization results are presented in Section \ref{sec:homo}. In the Appendix, certain properties of Moreau-Yosida approximations are recalled and Mosco convergence results for integral functionals are provided.

\subsection{Notation}

In the following we work with generic constants $C\ge0$, $c>0$ that are allowed to change value from line to line and we write 
\[
A\lesssim B
\]
if there is a constant $C\ge0$ such that $A\le CB$. If $(E,d)$ is a metric space, $R>0$ and $x\in E$, then $B_{R}(x)$ denotes the open ball of radius $R$ centered at $x$. We set 
\[
r^{[m]}:=|r|^{m-1}r\quad\forall r\in\R.
\]

We denote the $(d-1)$-dimensional unit sphere in $\R^{d}$ by $S^{d-1}$ and the volume of the unit ball in $\R^{d}$ by $\s_{d}$. Further, we let
\[
\sgn(\xi):=\begin{cases}
\frac{\xi}{|\xi|} & \text{if }\xi\ne0\\
\overline{B_{1}(0)} & \text{if }\xi=0,
\end{cases}
\]
be the maximal monotone extension of the sign function.

For $m\ge1$ we let $L^{m}(\mcO)$ be the usual Lebesgue spaces with norm $\|\cdot\|_{L^{m}}$ and we shall often use the shorthand notation $L^{m}:=L^{m}(\mcO),\|\cdot\|_{m}:=\|\cdot\|_{L^{m}(\mcO)}$. For a function $v\in L^{m}(\mcO)$ we define its extension to $\R^{d}$ by
\[
\bar{v}(\xi)=\begin{cases}
v(\xi) & \text{if }x\in\mcO\\
0 & \text{otherwise.}
\end{cases}
\]
and its average value on a bounded set $\mcO\subseteq\R^{d}$ by
\[
M_{\mcO}(v):=\frac{1}{|\mcO|}\int_{\mcO}v(\xi)\, d\xi,
\]
where $|\mcO|:=\int_{\mcO}\,d\xi$.
We further let $H^{k}=H^{k}(\mcO)=W^{2,k}(\mcO)$ be the usual Sobolev space of order $k\in\N$, $H_{0}^{1}$ be the space of functions in $H^{1}$ with trace zero on $\partial\mcO$ and $H^{-1}$ the Hilbert space dual of $H_{0}^{1}$. For $u\in L^{1}(\mcO)$ we define the total variation semi-norm by
\[
\|u\|_{TV}:=\sup\left\{ \int_{\mcO}u\div\eta\, d\xi\,:\,\eta\in C_{0}^{\infty}(\mcO;\R^{d}),\,\|\eta\|_{L^{\infty}}\le1\right\} 
\]
and let $BV$ be the space of functions of bounded variation, that is, 
\[
BV:=\{u\in L^{1}(\mcO):\,\|u\|_{TV}<\infty\}.
\]

We say that a function $X\in L^{1}([0,T]\times\O;H)$ is $\F_{t}$-progressively measurable if $X1_{[0,t]}$ is $\mcB([0,t])\otimes\mcF_{t}$-measurable for all $t\in[0,T]$.

\section{Generalities on stochastic variational inequalities\label{sec:generalities_SVI}}

Let $H$, $U$ be separable Hilbert spaces and let $L_{2}(U,H)$ denote the space of linear Hilbert-Schmidt operators from $U$ to $H$. Let $\{W_{t}\}_{t\ge0}$ be a cylindrical Wiener process on $U$ modeled on a normal filtered probability space $(\Omega,\Fcal,\{\Fcal_{t}\}_{t\ge0},\P)$. We consider the SPDE
\begin{equation}
dX_{t}\in-\partial\vp(X_{t})\, dt+B(X_{t})\, dW_{t},\label{eq:general_gradient_SPDE}
\end{equation}
where $\partial\vp$ is the subdifferential of a lower semi-continuous (l.s.c.), convex, proper function $\vp:H\to[0,+\infty]$. Without loss of generality we assume $\vp(0)=0$. Further, let $B:H\to L_{2}(U,H)$ be Lipschitz continuous diffusion coefficients, that is, there exists an $L>0$ such that for all $x,y\in H$
\begin{equation}
\norm{B(x)-B(y)}_{L_{2}(U,H)}\le L\norm{x-y}_{H}.\label{eq:B-is-lipschitz}
\end{equation}
Let $S$ be a separable Hilbert space continuously and densely embedded in $H$, that is, $S\hookrightarrow H$.
\begin{defn}
\label{def:SPDE_SVI}Let $x_{0}\in L^{2}(\Omega,\mcF_{0};H)$, $T>0.$ An $\Fcal_{t}$-progressively measurable map $X\in L^{2}([0,T]\times\Omega;H)$ is said to be an \emph{SVI solution} to \eqref{eq:general_gradient_SPDE} if there exists a $C>0$ such that
\begin{enumerate}
\item {[}Regularity{]} 
\begin{equation}
\esssup_{t\in[0,T]}\E\|X_{t}\|_{H}^{2}+\E\int_{0}^{T}\vp(X_{r})dr\le C(\E\|x_{0}\|_{H}^{2}+1).\label{eq:SVI_energy}
\end{equation}

\item {[}Variational inequality{]} For every admissible test-function $Z\in L^{2}([0,T]\times\O;S)$, that is, there are $Z_{0}\in L^{2}(\Omega,\mcF_{0};H)$, $G\in L^{2}([0,T]\times\O;H)$, $F\in L^{2}([0,T]\times\O;L_{2}(U,H))$ $\Fcal_{t}$-progressively measurable such that 
\begin{equation}
Z_{t}:=Z_{0}+\int_{0}^{t}G_{r}\, dr+\int_{0}^{t}F_{r}\, dW_{r}\quad\forall t\in[0,T],\label{eq:Z_decomp}
\end{equation}
we have that
\begin{align}
 & \E e^{-Ct}\|X_{t}-Z_{t}\|_{H}^{2}+2\E\int_{0}^{t}e^{-Cr}\vp(X_{r})\, dr\nonumber \\
 & \le\E\|x_{0}-Z_{0}\|_{H}^{2}+2\E\int_{0}^{t}e^{-Cr}\vp(Z_{r})\, dr-2\E\int_{0}^{t}e^{-Cr}(G_{r},X_{r}-Z_{r})_{H}\, dr\label{eq:SVI}\\
 & +2\E\int_{0}^{t}e^{-Cr}\|F_{r}-B(Z_{r})\|_{L_{2}}^{2}\, dr,\nonumber 
\end{align}
for almost all $t\in[0,T]$.
\end{enumerate}

If, additionally, $X\in L^{2}(\Omega;C([0,T];H))$, we say that $X$ is a \emph{(time-)continuous SVI solution} to \eqref{eq:general_gradient_SPDE}.

\end{defn}
Definition \ref{def:SPDE_SVI} modifies notions of stochastic SVI solutions introduced in \cite{BR13,BDPR09-4,GR15,GR14}. These modifications are chosen in order to obtain a stable notion of solutions with regard to approximations of the subdifferential $\vp$. More precisely,
\begin{rem}
\label{rem:SVI_notions}~
\begin{enumerate}
\item In \cite{BR13} SVI solutions are defined as time-continuous SVI solutions in the sense of Definition \ref{def:SPDE_SVI} but satisfying \eqref{eq:SVI} only for the special case $F=B(Z)$. The advantage of the (more restrictive) condition \eqref{eq:SVI} is its stability with respect to approximations of the test-functions $Z$. For example, if $P:S\to S$ is a continuous linear operator, then $PZ$ is again a valid test-function in \eqref{eq:SVI}, while it does not necessarily satisfy \eqref{eq:Z_decomp} with $F=B(PZ)$.
\item Definition \ref{def:SPDE_SVI} introduces non-time continuous SVI solutions, assuming only $X\in L^{2}([0,T]\times\Omega;H)$. The point of this generalization is that this property proves to be stable under random Mosco convergence $\vp^{n}\to\vp$ (cf.~Definition \ref{def:random_Mosco} below), while the continuity condition $X\in L^{2}(\Omega,C([0,T];H))$ does not. 
\end{enumerate}
\end{rem}

We say that an $\mcF_{t}$-adapted process $X\in L^{2}(\Omega,C([0,T];H))$ is a strong solution to \eqref{eq:general_gradient_SPDE}, if there exists an $\eta\in L^{2}([0,T]\times\O;H)$ progressively measurable such that $\eta\in\partial\vp(X)$ a.e.~and 
\[
X_{t}=x_{0}-\int_{0}^{t}\eta_{r}dr+\int_{0}^{t}B(X_{r})dW_{r}\quad\P\text{-a.s.}
\]
for all $t\ge0.$ 
\begin{rem}
\label{rmk:varn_sol}If $X$ is a strong solution to \eqref{eq:general_gradient_SPDE}, then $X$ is a time-continuous SVI solution to \eqref{eq:general_gradient_SPDE}. The constant $C$ in \eqref{eq:SVI_energy}, \eqref{eq:SVI} can be chosen depending on $L$, $\|B(0)\|_{L_{2}(U,H)}$ only, where $L$ is as in \eqref{eq:B-is-lipschitz}.\end{rem}
\begin{proof}
(i): By Itô's formula and a standard localization argument:
\[
\E\|X_{t}\|_{H}^{2}=\E\|x_{0}\|_{H}^{2}-2\E\int_{0}^{t}(\eta_{r},X_{r})_{H}dr+\E\int_{0}^{t}\|B(X_{r})\|_{L_{2}(U,H)}^{2}dr.
\]
By the definition of the subdifferential $\partial\vp$ we have that
\begin{align*}
(-\eta,X)_{H} & =(\eta,0-X)_{H}\le-\vp(X)\quad dt\otimes \P-\text{a.e.}
\end{align*}
and by Lipschitz continuity of $B$
\[
\|B(X_{r})\|_{L_{2}(U,H)}^{2}\le C(1+\|X_{r}\|_{H}^{2}).
\]
Hence,
\[
\E\|X_{t}\|_{H}^{2}+\E\int_{0}^{t}\vp(X_{r})dr\lesssim\E\|x_{0}\|_{H}^{2}+\E\int_{0}^{t}\|X_{r}\|_{H}^{2}dr+t.
\]
Gronwall's Lemma finishes the proof of \eqref{eq:SVI_energy}.

(ii): Let $Z\in L^{2}(\Omega;C([0,T];H))$ be given by 
\begin{align}
Z_{t} & =Z_{0}+\int_{0}^{t}G_{r}dr+\int_{0}^{t}F_{r}dW_{r}\label{eq:test-eq}
\end{align}
for some $Z_{0}\in L^{2}(\Omega,\mcF_{0};H)$, $G\in L^{2}([0,T]\times\O;H)$, $F\in L^{2}([0,T]\times\O;L_{2}(U,H))$ progressively measurable. Then
\begin{align*}
d(X_{t}-Z_{t}) & =(-\eta_{t}-G_{t})dt+(B(X_{t})-F_{t})dW_{t}
\end{align*}
and Itô's formula implies that
\begin{align*}
e^{-Ct}\|X_{t}-Z_{t}\|_{H}^{2}= & \|x_{0}-Z_{0}\|_{H}^{2}+2\int_{0}^{t}e^{-Cr}(-\eta_{r}-G_{r},X_{r}-Z_{r})_{H}\, dr\\
 & +2\int_{0}^{t}e^{-Cr}(X_{r}-Z_{r},B(X_{r})-F_{r})_{H}\, dW_{r}\\
 & +\int_{0}^{t}e^{-Cr}\|B(X_{r})-F_{r}\|_{L_{2}(U,H)}^{2}\, dr\\
 & -C\int_{0}^{t}e^{-Cr}\|X_{r}-Z_{r}\|_{H}^{2}\, dr\quad\forall t\in[0,T].
\end{align*}
Since
\begin{align*}
\|B(X_{r})-F_{r}\|_{L_{2}(U,H)}^{2} & \le2\|B(X_{r})-B(Z_{r})\|_{L_{2}(U,H)}^{2}+2\|B(Z_{r})-F_{r}\|_{L_{2}(U,H)}^{2}\\
 & \le2L^{2}\|X_{r}-Z_{r}\|_{H}^{2}+2\|B(Z_{r})-F_{r}\|_{L_{2}(U,H)}^{2}
\end{align*}
taking expectations and choosing $C\ge2L^{2}$ yields
\begin{align*}
e^{-Ct}\E\|X_{t}-Z_{t}\|_{H}^{2}= & \E\|x_{0}-Z_{0}\|_{H}^{2}+2\E\int_{0}^{t}e^{-Cr}(-\eta_{r}-G_{r},X_{r}-Z_{r})_{H}\, dr\\
 & +2\E\int_{0}^{t}e^{-Cr}\|B(Z_{r})-F_{r}\|_{L_{2}(U,H)}^{2}\, dr\quad\forall t\in[0,T].
\end{align*}
Since $\eta\in\partial\vp(X)$ a.e.~we have that
\[
(-\eta_{r},X_{r}-Z_{r})_{H}\le\vp(Z_{r})-\vp(X_{r}),\quad dt\otimes d\P-\text{a.e.}
\]
which finishes the proof.
\end{proof}
We next establish the stability of SVI solutions with respect to random Mosco convergence of convex functionals $\vp^{n}$ in the following sense
\begin{defn}
\label{def:random_Mosco}We say that $\vp^{n}\to\vp$ in \emph{random Mosco sense} if
\begin{enumerate}
\item For every sequence $Z^{n}\in L^{2}([0,T]\times\O;H)$ such that $Z^{n}\rightharpoonup Z$ for some $Z\in L^{2}([0,T]\times\O;H)$ and all $\g\in L^{\infty}([0,T])$ non-negative
\[
\liminf_{n\to\infty}\E\int_{0}^{T}\g_{r}\varphi^{n}(Z_{r}^{n})dr\ge\E\int_{0}^{T}\g_{r}\varphi(Z_{r})dr.
\]

\item For every admissible test-function $Z\in L^{2}([0,T]\times\O;S)$ with $\vp(Z)\in L^{1}([0,T]\times\O)$ and $dZ=Gdt+FdW$ there exists a sequence of admissible test-functions $Z^{n}\in L^{2}([0,T]\times\O;S)$ with $dZ^{n}=G^{n}dt+F^{n}dW$ such that
\begin{align*}
Z_{0}^{n} & \to Z_{0}\quad\text{in }L^{2}(\O;H)\\
G^{n} & \to G\quad\text{in }L^{2}([0,T]\times\O;H)\\
F^{n} & \to F\quad\text{in }L^{2}([0,T]\times\O;L_{2}(U,H))
\end{align*}
for $n\to\infty$ and, for all $\g\in L^{\infty}([0,T])$ non-negative,
\begin{equation}
\limsup_{n\to\infty}\E\int_{0}^{T}\g_{r}\varphi^{n}(Z_{r}^{n})dr\le\E\int_{0}^{T}\g_{r}\varphi(Z_{r})dr.\label{eq:2ndMosco}
\end{equation}

\end{enumerate}
\end{defn}
In Appendix \ref{sec:Random-Mosco-convergence} we show that $\vp^{n}\to\vp$ in Mosco sense implies that Definition \ref{def:random_Mosco}, (i) is satisfied. Hence, the additional structure required in order to deal with the presence of the stochastic perturbation in \eqref{eq:general_gradient_SPDE} is reflected by Definition \ref{def:random_Mosco}, (ii) only. As it turns out, this property is easily verified in applications based on the following proposition.
\begin{prop}
\label{prop:random_Mosco}Let $\vp^{n}$, $\vp$ be convex, l.s.c., proper functions on $H$, such that $\vp^{n}\to\vp$ in Mosco sense. Suppose either of the following
\begin{enumerate}
\item For all $Z\in L^{2}([0,T]\times\O;S)$ and all $\g\in L^{\infty}([0,T])$ non-negative
\begin{equation}
\limsup_{n\to\infty}\E\int_{0}^{T}\g_{r}\varphi^{n}(Z)dr\le\E\int_{0}^{T}\g_{r}\varphi(Z)dr.\label{eq:2ndMosco-1}
\end{equation}

\item For some $C>0$,
\[
\limsup_{n\to\infty}\vp^{n}(u)\le\vp(u)\quad\forall u\in S
\]
and 
\begin{equation}
\vp^{n}(u)\le C(1+\vp(u)+\norm{u}_{S}^{2})\quad\forall u\in S.\label{eq:bound_for_vp-2}
\end{equation}

\end{enumerate}
Then $\vp^{n}\to\vp$ in random Mosco sense.\end{prop}
\begin{proof}
(i): Obvious, choosing $Z^{n}\equiv Z$ in Definition \ref{def:random_Mosco}, (ii). 

(ii): By the reverse Fatou's inequality, using the bound \eqref{eq:bound_for_vp-2}, we obtain
\[
\limsup_{n\to\infty}\E\int_{0}^{T}\g_{r}\varphi^{n}(Z)dr\le\E\int_{0}^{T}\g_{r}\varphi(Z)dr,
\]
which concludes the proof by (i). 
\end{proof}
For example, let $\vp^{n}$ be the Moreau-Yosida approximation of a convex, l.s.c., proper function $\vp:H\to[0,\infty]$. Then Proposition \ref{prop:random_Mosco} implies that $\vp^{n}\to\vp$ in random Mosco sense, cf.~Proposition \ref{prop:ex-svi-sln} below.

We have the following general stability property of SVI solutions with respect to random Mosco convergence:
\begin{thm}
\label{thm:SVI_stability}Let $x_{0}\in L^{2}(\O,\mcF_{0};H)$ and $\vp^{n}$ be a sequence of convex, l.s.c., proper functions such that $\vp^{n}\to\vp$ in random Mosco sense. Let $X^{n}$ be SVI solutions to \eqref{eq:general_gradient_SPDE} for $\vp$ replaced by $\vp^{n}$ satisfying \eqref{eq:SVI_energy}, \eqref{eq:SVI} with a constant $C>0$ independent of $n$. Then there is an SVI solution $X$ to \eqref{eq:general_gradient_SPDE} and a subsequence $X^{n_{k}}$ such that 
\[
X^{n_{k}}\rightharpoonup X\quad\text{in }L^{2}([0,T]\times\O;H).
\]
If SVI solutions to \eqref{eq:general_gradient_SPDE} are unique, then the whole sequence $X^{n}$ converges weakly to $X$.\end{thm}
\begin{proof}
By property Definition \ref{def:SPDE_SVI}, (i): 
\begin{align*}
\esssup_{t\in[0,T]}\E\|X_{t}^{n}\|_{H}^{2}+\E\int_{0}^{T}\vp^{n}(X_{r}^{n})dr & \le C(\E\|x_{0}\|_{H}^{2}+1)<\infty.
\end{align*}
Therefore, for a subsequence
\[
X^{n_{k}}\rightharpoonup X\quad\text{in }L^{2}([0,T]\times\O;H),
\]
for some progressively measurable $X\in L^{2}([0,T]\times\O;H)$. Since $\vp^{n}\to\vp$ in random Mosco sense, we obtain that 
\begin{equation}
\liminf_{n\to\infty}\E\int_{0}^{T}\g_{r}\vp^{n_{k}}(X_{r}^{n_{k}})\, dr\ge\E\int_{0}^{T}\g_{r}\vp(X_{r})\, dr\label{eq:liminf_1}
\end{equation}
for all $\g\in L^{\infty}([0,T])$ non-negative. Hence,
\begin{align*}
\esssup_{t\in[0,T]}\E\|X_{t}\|_{H}^{2}+\E\int_{0}^{T}\vp(X_{r})dr & \le C(\E\|x_{0}\|_{H}^{2}+1).
\end{align*}
It remains to prove that $X$ satisfies \eqref{eq:SVI}. Let $Z\in L^{2}([0,T]\times\O;S)$ with $\vp(Z)\in L^{1}([0,T]\times\O)$ and satisfying \eqref{eq:Z_decomp} for some $Z_{0}\in L^{2}(\O,\mcF_{0};H)$, $G\in L^{2}([0,T]\times\O;H)$, $F\in L^{2}([0,T]\times\O;L_{2}(U,H))$ progressively measurable. By random Mosco convergence there exist sequences $Z_{0}^{n}\in L^{2}(\O,\mcF_{0};H)$, $G^{n}\in L^{2}([0,T]\times\O;H)$, $F^{n}\in L^{2}([0,T]\times\O;L_{2}(U,H))$ progressively measurable such that 
\begin{align*}
Z_{0}^{n} & \to Z\quad\text{in }L^{2}(\O;H)\\
G^{n} & \to G\quad\text{in }L^{2}([0,T]\times\O;H)\\
F^{n} & \to F\quad\text{in }L^{2}([0,T]\times\O;L_{2}(U,H))
\end{align*}
and
\begin{equation}
\limsup_{n\to\infty}\E\int_{0}^{T}\g_{r}\vp^{n}(Z_{r}^{n})dr\le\E\int_{0}^{T}\g_{r}\vp(Z_{r})dr,\label{eq:limsup}
\end{equation}
for all $\g\in L^{\infty}([0,T])$ non-negative. Clearly,
\[
Z_{t}^{n}:=Z_{0}^{n}+\int_{0}^{t}G_{s}^{n}\, ds+\int_{0}^{t}F_{s}^{n}\, dW_{s}\to Z\quad\text{in }L^{2}([0,T]\times\O;H).
\]
Since $X^{n_{k}}$ is an SVI solution we have that 
\begin{align}
 & \E e^{-Ct}\|X_{t}^{n_{k}}-Z_{t}^{n_{k}}\|_{H}^{2}+2\E\int_{0}^{t}e^{-Cr}\vp^{n_{k}}(X_{r}^{n_{k}})\, dr\nonumber \\
 & \le\E\|x_{0}-Z_{0}^{n}\|_{H}^{2}+2\E\int_{0}^{t}e^{-Cr}\vp^{n_{k}}(Z_{r}^{n_{k}})\, dr\label{eq:SVI-1}\\
 & -2\E\int_{0}^{t}e^{-Cr}(G_{r}^{n_{k}},X_{r}^{n_{k}}-Z_{r}^{n_{k}})_{H}\, dr\nonumber \\
 & +2\E\int_{0}^{t}e^{-Cr}\|F_{r}^{n_{k}}-B(Z_{r}^{n_{k}})\|_{L_{2}}^{2}\, dr\quad\text{for a.e. }t\in[0,T].\nonumber 
\end{align}
By \eqref{eq:liminf_1} and Fatou's Lemma for each $\g\in L^{\infty}([0,T])$ non-negative we have that 
\begin{align*}
\liminf_{k\to\infty}\int_{0}^{T}\g_{t}\E\int_{0}^{t}e^{-Cr}\vp^{n_{k}}(X_{r}^{n_{k}})\, drdt & \ge\int_{0}^{T}\g_{t}\liminf_{k\to\infty}\E\int_{0}^{t}e^{-Cr}\vp^{n_{k}}(X_{r}^{n_{k}})\, drdt\\
 & \ge\int_{0}^{T}\g_{t}\E\int_{0}^{t}e^{-Cr}\vp(X_{r})\, drdt.
\end{align*}
Moreover, since
\[
\E\int_{0}^{t}e^{-Cr}\vp^{n_{k}}(Z_{r}^{n_{k}})\, dr\le\E\int_{0}^{T}e^{-Cr}\vp^{n_{k}}(Z_{r}^{n_{k}})\, dr\quad\forall t\in[0,T],
\]
we can apply the reverse Fatou's Lemma and \eqref{eq:limsup} to obtain that 
\begin{align*}
\limsup_{n\to\infty}\int_{0}^{T}\g_{t}\E\int_{0}^{t}e^{-Cr}\vp^{n_{k}}(Z_{r}^{n_{k}})\, drdt & \le\int_{0}^{T}\limsup_{n\to\infty}\g_{t}\E\int_{0}^{t}e^{-Cr}\vp^{n_{k}}(Z_{r}^{n_{k}})\, drdt\\
 & \le\int_{0}^{T}\g_{t}\E\int_{0}^{t}e^{-Cr}\vp(Z_{r})\, drdt.
\end{align*}
Hence, integrating \eqref{eq:SVI-1} against $\g\in L^{\infty}([0,T])$ non-negative and taking $\liminf_{k\to\infty}$ we obtain that 
\begin{align}
 & \int_{0}^{T}\g_{t}\E e^{-Ct}\|X_{t}-Z_{t}\|_{H}^{2}dt+2\int_{0}^{T}\g_{t}\E\int_{0}^{t}e^{-Cr}\vp(X_{r})\, drdt\nonumber \\
 & \le\int_{0}^{T}\g_{t}\E\|x_{0}-Z_{0}\|_{H}^{2}dt+2\int_{0}^{T}\g_{t}\E\int_{0}^{t}e^{-Cr}\vp(Z_{r})\, drdt\label{eq:SVI-1-1}\\
 & -2\int_{0}^{T}\g_{t}\E\int_{0}^{t}e^{-Cr}(G_{r},X_{r}-Z_{r})_{H}\, drdt\nonumber \\
 & +2\int_{0}^{T}\g_{t}\E\int_{0}^{t}e^{-Cr}\|F_{r}-B(Z_{r})\|_{L_{2}}^{2}\, drdt.\nonumber 
\end{align}
Since this is true for all $\g\in L^{\infty}([0,T])$ non-negative, the claim follows.
\end{proof}
The same proof as for Theorem \ref{thm:SVI_stability} also allows to study perturbations of the diffusion coefficients $B$. More precisely,
\begin{rem}
In the situation of Theorem \ref{thm:SVI_stability} let $B^{n}:H\to L_{2}(U,H)$ be uniformly Lipschitz continuous, that is, satisfy \eqref{eq:B-is-lipschitz} with a constant $L$ independent of $n$, and 
\[
B^{n}(u)\to B(u)\quad\text{in }L_{2}(U,H)
\]
for all $u\in H$. Let $X^{n}$ be SVI solutions to \eqref{eq:general_gradient_SPDE} for $\vp$ replaced by $\vp^{n}$ and $B$ replaced by $B^{n}$ satisfying \eqref{eq:SVI_energy}, \eqref{eq:SVI} with a constant $C>0$ independent of $n$. Then there is an SVI solution $X$ to \eqref{eq:general_gradient_SPDE} and a subsequence $X^{n_{k}}$ such that 
\[
X^{n_{k}}\rightharpoonup X\quad\text{in }L^{2}([0,T]\times\O;H).
\]
If SVI solutions to \eqref{eq:general_gradient_SPDE} are unique, then the whole sequence $X^{n}$ converges weakly to $X$.\end{rem}
\begin{proof}
We follow the proof of Theorem \ref{thm:SVI_stability}, observing that
\begin{align*}
 & \|F_{r}^{n_{k}}-B^{n_{k}}(Z_{r}^{n_{k}})\|_{L_{2}}^{2}\\
 & \le2\|F_{r}^{n_{k}}-B(Z_{r})\|_{L_{2}}^{2}+2\|B^{n_{k}}(Z_{r}^{n_{k}})-B(Z_{r})\|_{L_{2}}^{2}\\
 & \le2\|F_{r}^{n_{k}}-B(Z_{r})\|_{L_{2}}^{2}+2\|B^{n_{k}}(Z_{r}^{n_{k}})-B^{n_{k}}(Z_{r})\|_{L_{2}}^{2}+2\|B^{n_{k}}(Z_{r})-B(Z_{r})\|_{L_{2}}^{2}\\
 & \le2\|F_{r}^{n_{k}}-B(Z_{r})\|_{L_{2}}^{2}+2L^{2}\|Z_{r}^{n_{k}}-Z_{r}\|_{H}^{2}+2\|B^{n_{k}}(Z_{r})-B(Z_{r})\|_{L_{2}}^{2}.
\end{align*}
Since $B^{n}$ is Lipschitz continuous and pointwise convergent to $B$ we have that
\[
\|B^{n}(u)\|_{L_{2}}\le C(1+\|u\|_{H})\quad\forall u\in H
\]
with a constant $C>0$ independent of $n$. Hence, by dominated convergence $\|B^{n}(Z_{r})-B(Z_{r})\|_{L_{2}}^{2}\to0$ for $n\to\infty$ and the proof can be finished as before. \end{proof}
\begin{prop}
\label{prop:ex-svi-sln}Let $\vp$ be a l.s.c., convex, proper function on $H$ and let $x_{0}\in L^{2}(\O,\mcF_{0};H)$. Then:
\begin{enumerate}
\item There is an SVI solution $X$ to \eqref{eq:general_gradient_SPDE}.
\item The set of SVI solutions to \eqref{eq:general_gradient_SPDE} satisfying \eqref{eq:SVI_energy}, \eqref{eq:SVI} with a uniform $C>0$ is non-empty, convex and closed in \textup{$L^{2}([0,T]\times\O;H)$. }
\end{enumerate}
\end{prop}
\begin{proof}
(i): We consider the Moreau-Yosida approximation $\vp^{n}$ of $\vp$. Then $\partial\vp^{n}$ is single-valued and Lipschitz continuous (cf.~e.g.~\cite{B10}). It is easy to see that 
\begin{align}
dX_{t}^{n} & =-\partial\vp^{n}(X_{t}^{n})\, dt+B(X_{t}^{n})\, dW_{t}\label{eq:yosida-approx-SPDE}\\
X_{0}^{n} & =x_{0}\nonumber 
\end{align}
has a unique, strong solution $X^{n}\in L^{2}(\O;C([0,T];H))$. Thus, $X^{n}$ is also an SVI solution to \eqref{eq:yosida-approx-SPDE}. Moreover, $\vp^{n}\to\vp$ in Mosco- and in pointwise sense and $\vp^{n}\le\vp$ (cf.~e.g.~\cite{B10}). By Proposition \ref{prop:random_Mosco} (ii) this implies that $\vp^{n}\to\vp$ in random Mosco sense. Hence, by Theorem \ref{thm:SVI_stability}  there is an SVI solution for $\vp$.

(ii): Convexity follows from convexity of $\|\cdot\|_{H}^{2}$ and $\vp$. Non-emptiness follows from (i). Closedness follows from Theorem \ref{thm:SVI_stability}.
\end{proof}

\section{SVI approach to stochastic $p$-Laplace equations\label{sec:plp}}

In this section we develop an SVI approach to stochastic singular $p$-Laplace evolution equations with zero Neumann boundary conditions, that is, SPDE of the type
\begin{align}
dX_{t} & \in\div\left(|\nabla X_{t}|^{p-2}\nabla X_{t}\right)\, dt+B(X_{t})\, dW_{t},\nonumber \\
|\nabla X_{t}|^{p-2}\nabla X_{t}\cdot\nu & \ni0\quad\mathrm{on}\;\partial\Ocal,\; t>0,\label{eq:singular_p_laplace-1}\\
X_{0} & =x_{0}\nonumber 
\end{align}
on bounded, convex, smooth domains $\mcO\subseteq\R^{d}$ and with $p\in[1,2)$, where $\nu$ denotes the outer normal on $\partial\Ocal$. In particular, we include the multi-valued case $p=1$ for which we set $|r|^{-1}r=\sgn(r)$, the multi-valued extension of the sign function. In the following we will work with the Hilbert spaces $H=L^{2}(\mcO)$, $S=H^{1}(\mcO)$ and the Banach space $V=(W^{1,p}\cap L^{2})(\mcO)$. We suppose that $B$ satisfies the following assumptions
\begin{enumerate}
\item[(B)] There exists a $C>0$ such that
\begin{equation}
\|B(v)-B(w)\|_{L_{2}(U,H)}^{2}\le C\|v-w\|_{H}^{2}\quad\forall v,w\in H\label{eq:lip-lip}
\end{equation}
and 
\begin{equation}
\|B(v)\|_{L_{2}(U,S)}^{2}\le C(1+\|v\|_{S}^{2})\quad\forall v\in S.\label{eq:noise_h01}
\end{equation}

\end{enumerate}
Let $\psi(\xi)=\frac{1}{p}|\xi|^{p}$ and $\phi(\xi)=\partial\psi(\xi)=|\xi|^{p-2}\xi$. We define, for $p\in(1,2)$,
\[
\vp(v):=\begin{cases}
\int_{\mcO}\psi(\nabla v)\, d\xi & \text{if }v\in(W^{1,p}\cap L^{2})(\mcO)\\
+\infty & \text{if }v\in L^{2}(\mcO)\setminus W^{1,p}(\mcO)
\end{cases}
\]
and for $p=1$,
\[
\vp(v):=\begin{cases}
\|v\|_{TV} & \text{if }v\in(BV\cap L^{2})(\mcO)\\
+\infty & \text{if }v\in L^{2}(\mcO)\setminus BV(\mcO).
\end{cases}
\]
Obviously, $\vp$ is convex and it is easy to see that $\vp$ is lower-semicontinuous on $H$. Since $\vp$ is the lower-semicontinuous hull of $\vp_{|H^{1}}$ on $H$, for $u\in H^{1}$ we have that 
\[
\{-\div\eta:\,\eta\in H^{1},\,\eta\in\phi(\nabla u),\, d\xi\text{-a.e. and }\eta\cdot\nu=0\text{ a.e. on }\partial\mcO\}\subseteq\partial\vp(u).
\]

Hence, we may rewrite \eqref{eq:singular_p_laplace-1} in the relaxed form
\begin{align}
dX_{t} & \in-\partial\vp(X_{t})\, dt+B(X_{t})\, dW_{t},\label{eq:singular_p_laplace-2}\\
X_{0} & =x_{0}\nonumber 
\end{align}
and Definition \ref{def:SPDE_SVI} yields the concept of (continuous) SVI solutions to \eqref{eq:singular_p_laplace-1}. 

We note that, if $p>1$, solutions to \eqref{eq:singular_p_laplace-1} have been constructed in \cite{RRW07} by variational methods. In order to prove convergence of nonlocal approximations we require the weaker notion of SVI solutions. In particular, we will prove uniqueness of SVI solutions to \eqref{eq:singular_p_laplace-1} which is a stronger uniqueness result than previously known. 

The case $p=1$, the stochastic total variation flow, has been recently considered in \cite{BR13}, where well-posedness of SVI solutions to \eqref{eq:singular_p_laplace-1} in the case of linear multiplicative noise has been shown, by means of a different method. We extend this well-posedness result to general multiplicative noise. In addition, our results complement those of \cite{GT11} by characterizing the limit solutions constructed in \cite{GT11} as SVI solutions to \eqref{eq:singular_p_laplace-1}.

The main result of the current section is the proof of well-posedness of  \eqref{eq:singular_p_laplace-2} in the sense of Definition \ref{def:SPDE_SVI}.
\begin{thm}
\label{thm:SVI}Let $x_{0}\in L^{2}(\Omega,\mcF_{0};H).$ Suppose that \eqref{eq:lip-lip} and \eqref{eq:noise_h01} are satisfied. Then there is a unique continuous SVI solution $X\in L^{2}(\Omega;C([0,T];H))$ to \eqref{eq:singular_p_laplace-2} in the sense of Definition \ref{def:SPDE_SVI}. For two SVI solutions $X$, $Y$ with initial conditions $x_{0},y_{0}\in L^{2}(\Omega;H)$ we have
\[
\esssup_{t\in[0,T]}\E\|X_{t}-Y_{t}\|_{H}^{2}\lesssim\E\|x_{0}-y_{0}\|_{H}^{2}.
\]
\end{thm}
\begin{proof}
\emph{} The proof is based on a three step approximation of \eqref{eq:singular_p_laplace-1}. Let $\psi,\psi^{\d},\phi^{\d},R{}_{\d}$ be as in Appendix \ref{sec:Moreau-Yosida}, $x_{0}^{n}\to x_{0}$ in $L^{2}(\O;H)$ with $x_{0}^{n}\in L^{2}(\O,\mcF_{0};H^{1})$ and $\ve>0$. We then consider the non-degenerate, non-singular approximating SPDE
\begin{align}
dX_{t}^{\ve,\d,n} & =\ve\D X_{t}^{\ve,\d,n}\, dt+\div\phi^{\d}\left(\nabla X_{t}^{\ve,\d,n}\right)\, dt+B(X_{t}^{\ve,\d,n})\, dW_{t},\label{eq:full_approx_SVI_constr}\\
X_{0}^{\ve,\d,n} & =x_{0}^{n},\nonumber 
\end{align}
with zero Neumann boundary conditions. We will first establish the existence of strong solutions to \eqref{eq:full_approx_SVI_constr} and then prove their convergence in the singular, degenerate limit $\d\to0$, $\ve\to0$, $n\to\infty$. 

\emph{Step 1: Non-singular, non-degenerate approximation.}

In this step we consider \eqref{eq:full_approx_SVI_constr} for $\d,\ve>0$, $n\in\N$ fix. We thus suppress them in the notation of $X^{\ve,\d,n}$ and $\phi^{\d}$. By \cite{PR07} there is a unique variational solution $X$ to \eqref{eq:full_approx_SVI_constr} with respect to the Gelfand triple $H^{1}\hookrightarrow L^{2}\hookrightarrow(H^{1})^{*}$ satisfying
\[
\E\sup_{t\in[0,T]}\|X_{t}\|_{H}^{2}\le C(\E\|x_{0}\|_{H}^{2}+1).
\]

\textit{Claim}: We have
\begin{equation}
\E\sup_{t\in[0,T]}\|X_{t}\|_{H^{1}}^{2}+2\ve\E\int_{0}^{T}\|\D X_{r}\|_{H}^{2}dr\le C(\E\|x_{0}\|_{H^{1}}^{2}+1),\label{eq:strong_soln_visc}
\end{equation}
with a constant $C>0$ independent of $\ve$, $\d$ and $n$.

Indeed: In the following we let $(e_{i})_{i=1}^{\infty}$ be an orthonormal basis of eigenvectors of the Neumann Laplacian $-\D$ on $L^{2}(\Ocal)$. We further let $P_{n}:H\to\text{span}\{e_{1},\dots,e_{n}\}$ be the orthogonal projection onto the span of the first $n$ eigenvectors. We recall that the unique variational solution $X^{\ve}$ to \eqref{eq:full_approx_SVI_constr} is constructed in \cite{PR07} as the (weak) limit $X$ of the following Galerkin approximation
\begin{align*}
dX_{t}^{n} & =\ve P_{n}\D X_{t}^{n}\, dt+P_{n}\div\phi(\nabla X_{t}^{n})\, dt+P_{n}B(X_{t}^{n})\, dW_{t}^{n},\\
X_{0}^{n} & =P_{n}x_{0}.
\end{align*}
By \cite[Theorem 4.2.4 and its proof]{PR07}, $X^n\rightharpoonup X$ weakly in $L^{2}([0,T]\times\O;H)$, $X$ is unique and $X\in L^2(\Omega;C([0,T];H))$.
We set $\|v\|_{\dot{H}^{1}}^{2}:=\|\nabla v\|_{2}^{2}$ for $v\in H^{1}$. Itô's formula then yields
\begin{align*}
\|X_{t}^{n}\|_{\dot{H}^{1}}^{2} & =\|P_{n}x_{0}\|_{\dot{H}^{1}}^{2}+2\int_{0}^{t}(X_{r}^{n},\ve P_{n}\D X_{r}^{n}+P_{n}\div\phi(\nabla X_{r}^{n}))_{\dot{H}^{1}}\, dr\\
 & +2\int_{0}^{t}(X_{r}^{n},P_{n}B(X_{r}^{n})\, dW_{r}^{n})_{\dot{H}^{1}}\, dr+\int_{0}^{t}\|P_{n}B(X_{r}^{n})\|_{L_{2}(U,\dot{H}^{1})}^{2}\, dr\\
 & =\|P_{n}x_{0}\|_{\dot{H}^{1}}^{2}-2\ve\int_{0}^{t}\|\D X_{r}^{n}\|_{H}^{2}\, dr+2\int_{0}^{t}(X_{r}^{n},P_{n}\div\phi(\nabla X_{r}^{n}))_{\dot{H}^{1}}\, dr\\
 & +2\int_{0}^{t}(X_{r}^{n},P_{n}B(X_{r}^{n})\, dW_{r}^{n})_{\dot{H}^{1}}\, dr+\int_{0}^{t}\|P_{n}B(X_{r}^{n})\|_{L_{2}(U,\dot{H}^{1})}^{2}\, dr.
\end{align*}
For $v\in H^{2}$ with $\phi(\nabla v)\cdot\nu=0$ on $\partial\mcO$, arguing as in \cite[Example 7.11]{GT11}, we obtain that 
\begin{align}
(v,\div\phi(\nabla v))_{\dot{H}^{1}} & =(-\D v,\div\phi(\nabla v))_{H}\nonumber \\
 & =\lim_{n\to\infty}(T_{n}v,\div\phi(\nabla v))_{H}\nonumber \\
 & =\lim_{n\to\infty}(nu-nJ_{n}u,\div\phi(\nabla v))_{H}\label{eq:liu-ref}\\
 & \le\lim_{n\to\infty}n\left(\int_{\mcO}\psi(\nabla J_{n}u)d\xi-\int_{\mcO}\psi(\nabla u)d\xi\right)\nonumber \\
 & \le0\nonumber 
\end{align}
where $T_{n}$ is the Yosida-approximation and $J_{n}$ the resolvent of the Neumann Laplacian $-\D$ on $L^{2}$. Using this, \eqref{eq:noise_h01} and the Burkholder-Davis-Gundy inequality yields
\begin{align}
\frac{1}{2}\E\sup_{t\in[0,T]}e^{-Ct}\|X_{t}^{n}\|_{H^{1}}^{2} & \le\E\|x_{0}\|_{H^{1}}^{2}-2\ve\E\int_{0}^{T}e^{-Cr}\|\D X_{r}^{n}\|_{H}^{2}\, dr+C,\label{eq:galerkin_ineq}
\end{align}
for some $C>0$ large enough. Hence, $X^{n}$ is uniformly bounded in $L^{2}([0,T]\times\O;H^{2})$ and $L^{2}(\O;L^{\infty}([0,T];H^{1}))$ and we may extract a weakly (weak$^{\ast}$ resp.) convergent subsequence (for simplicity we stick with the notation $X^{n}$). Therefore, we have
\begin{align*}
X^{n} & \rightharpoonup X,\quad\text{in }L^{2}([0,T]\times\O;H^{2}),\\
X^{n} & \rightharpoonup^{*}X,\quad\text{in }L^{2}(\O;L^{\infty}([0,T];H^{1})),
\end{align*}
for $n\to\infty$. Here, $X\in L^2(\Omega;C([0,T];H))$ is as above. By weak lower semicontinuity of the norms we may pass to the limit in \eqref{eq:galerkin_ineq} which yields the claim.

\emph{Step 2: Singular limit ($\d\to0$).} In this step we consider the singular limit $\d\to0$. Since we keep $\ve,n$ fix they are suppressed in the notation. Let $X^{\d}$ be the strong solution to \eqref{eq:full_approx_SVI_constr} constructed in step one. For two solutions $X^{\d_{1}},X^{\d_{2}}$ to \eqref{eq:full_approx_SVI_constr} with initial condition $x_{0}\in L^{2}(\O;H^{1})$ we have
\begin{align*}
e^{-Kt}\|X_{t}^{\d_{1}}-X_{t}^{\d_{2}}\|_{H}^{2}= & 2\int_{0}^{t}e^{-Kr}(\ve\D X_{r}^{\d_{1}}-\ve\D X_{r}^{\d_{2}},X_{r}^{\d_{1}}-X_{r}^{\d_{2}})_{H}\, dr\\
 & +2\int_{0}^{t}e^{-Kr}(\div\phi^{\d_{1}}(\nabla X_{r}^{\d_{1}})-\div\phi^{\d_{2}}(\nabla X_{r}^{\d_{2}}),X_{r}^{\d_{1}}-X_{r}^{\d_{2}})_{H}\, dr\\
 & +2\int_{0}^{t}e^{-Kr}(X_{r}^{\d_{1}}-X_{r}^{\d_{2}},B(X_{r}^{\d_{1}})-B(X_{r}^{\d_{2}}))_{H}\, dW_{r}\\
 & +\int_{0}^{t}e^{-Kr}\|B(X_{r}^{\d_{1}})-B(X_{r}^{\d_{2}})\|_{L_{2}}^{2}\, dr\\
 & -K\int_{0}^{t}e^{-Kr}\|X_{r}^{\d_{1}}-X_{r}^{\d_{2}}\|_{H}^{2}\, dr.
\end{align*}
Due to \eqref{eq:monotone_Y_bound} we observe that
\begin{align*}
 & (\div\phi^{\d_{1}}(\nabla X_{r}^{\d_{1}})-\div\phi^{\d_{2}}(\nabla X_{r}^{\d_{2}}),X_{r}^{\d_{1}}-X_{r}^{\d_{2}})_{H}\\
 & =-\int_{\mcO}(\phi^{\d_{1}}(\nabla X_{r}^{\d_{1}})-\phi^{\d_{2}}(\nabla X_{r}^{\d_{2}}))\cdot(\nabla X_{r}^{\d_{1}}-\nabla X_{r}^{\d_{2}})\, d\xi\\
 & \le C(\d_{1}+\d_{2})\int_{\mcO}(1+|\nabla X_{r}^{\d_{1}}|^{2}+|\nabla X_{r}^{\d_{2}}|^{2})\, d\xi\\
 & \le C(\d_{1}+\d_{2})(1+\|X_{r}^{\d_{1}}\|_{H^{1}}^{2}+\|X_{r}^{\d_{2}}\|_{H^{1}}^{2}).
\end{align*}
$dr\otimes\P$-a.e.. Moreover,
\begin{align*}
(\ve\D X_{r}^{\d_{1}}-\ve\D X_{r}^{\d_{2}},X_{r}^{\d_{1}}-X_{r}^{\d_{2}})_{H} & \le0
\end{align*}
$dr\otimes\P$-a.e.. Thus,
\begin{align*}
e^{-Kt}\|X_{t}^{\d_{1}}-X_{t}^{\d_{2}}\|_{H}^{2}\le & C(\d_{1}+\d_{2})\int_{0}^{t}(1+\|X_{r}^{\d_{1}}\|_{H^{1}}^{2}+\|X_{r}^{\d_{2}}\|_{H^{1}}^{2})\, dr\\
 & +2\int_{0}^{t}e^{-Kr}(X_{r}^{\d_{1}}-X_{r}^{\d_{2}},B(X_{r}^{\d_{1}})-B(X_{r}^{\d_{2}}))_{H}\, dW_{r}\\
 & +C\int_{0}^{t}e^{-Kr}\|X_{r}^{\d_{1}}-X_{r}^{\d_{2}}\|_{H}^{2}\, dr\\
 & -K\int_{0}^{t}e^{-Kr}\|X_{r}^{\d_{1}}-X_{r}^{\d_{2}}\|_{H}^{2}\, dr.
\end{align*}
Using the Burkholder-Davis-Gundy inequality and \eqref{eq:strong_soln_visc} we obtain
\begin{align}
\E\sup_{t\in[0,T]}e^{-Kt}\|X_{t}^{\d_{1}}-X_{t}^{\d_{2}}\|_{H}^{2}\le & C(\d_{1}+\d_{2})(\E\|x_{0}\|_{H^{1}}^{2}+1),\label{eq:stability_ic_SFDE-2}
\end{align}
for $K>0$ large enough. Hence, we obtain the existence of an $\{\mcF_{t}\}$-adapted process $X\in L^{2}(\O;C([0,T];H))$ with $X_{0}=x_{0}$ such that
\[
\E\sup_{t\in[0,T]}\|X_{t}^{\d}-X_{t}\|_{H}^{2}\to0\quad\text{for }\d\to0.
\]

\emph{Step 3: Vanishing viscosity ($\ve\to0$).} For two solutions $X^{\ve_{1},\d},X^{\ve_{2},\d}$ to \eqref{eq:full_approx_SVI_constr} with initial conditions $x_{0}^{1},x_{0}^{2}\in L^{2}(\O;H_{}^{1})$ we have
\begin{align*}
 & e^{-Kt}\|X_{t}^{\ve_{1},\d}-X_{t}^{\ve_{2},\d}\|_{H}^{2}\\
 & =\|x_{0}^{1}-x_{0}^{2}\|_{H}^{2}+2\int_{0}^{t}e^{-Kr}(\ve_{1}\D X_{r}^{\ve_{1},\d}-\ve_{2}\D X_{r}^{\ve_{2},\d},X_{r}^{\ve_{1},\d}-X_{r}^{\ve_{2},\d})_{H}\, dr\\
 & +2\int_{0}^{t}e^{-Kr}(\div\phi^{\d}(\nabla X_{r}^{\ve_{1},\d})-\div\phi^{\d}(\nabla X_{r}^{\ve_{2},\d}),X_{r}^{\ve_{1},\d}-X_{r}^{\ve_{2},\d})_{H}\, dr\\
 & +2\int_{0}^{t}e^{-Kr}(X_{r}^{\ve_{1},\d}-X_{r}^{\ve_{2},\d},B(X_{r}^{\ve_{1},\d})-B(X_{r}^{\ve_{2},\d}))_{H}\, dW_{r}\\
 & +\int_{0}^{t}e^{-Kr}\|B(X_{r}^{\ve_{1},\d})-B(X_{r}^{\ve_{2},\d})\|_{L_{2}}^{2}\, dr\\
 & -K\int_{0}^{t}e^{-Kr}\|X_{r}^{\ve_{1},\d}-X_{r}^{\ve_{2},\d}\|_{H}^{2}\, dr.
\end{align*}
We note
\begin{align*}
(\phi^{\d}(a)-\phi^{\d}(b))\cdot(a-b)\ge & 0\quad\forall a,b\in\R^{d}
\end{align*}
and
\begin{align*}
 & (\ve_{1}\D X_{r}^{\ve_{1},\d}-\ve_{2}\D X_{r}^{\ve_{2},\d},X_{r}^{\ve_{1},\d}-X_{r}^{\ve_{2},\d})_{H}\\
 & =\int_{\mcO}(\ve_{1}\nabla X_{r}^{\ve_{1},\d}-\ve_{2}\nabla X_{r}^{\ve_{2},\d})\cdot(\nabla X_{r}^{\ve_{1},\d}-\nabla X_{r}^{\ve_{2},\d})\, d\xi\\
 & \le C(\ve_{1}+\ve_{2})(\|X_{r}^{\ve_{1},\d}\|_{H^{1}}^{2}+\|X_{r}^{\ve_{2},\d}\|_{H^{1}}^{2}),
\end{align*}
$dt\otimes\P$-a.e.. Thus,
\begin{align*}
e^{-Kt}\|X_{t}^{\ve_{1},\d}-X_{t}^{\ve_{2},\d}\|_{H}^{2}\le & \|x_{0}^{1}-x_{0}^{2}\|_{H}^{2}\\
 & +C(\ve_{1}+\ve_{2})\int_{0}^{t}(1+\|X_{r}^{\ve_{1},\d}\|_{H^{1}}^{2}+\|X_{r}^{\ve_{2},\d}\|_{H^{1}}^{2})\, dr\\
 & +2\int_{0}^{t}e^{-Kr}(X_{r}^{\ve_{1},\d}-X_{r}^{\ve_{2},\d},B(X_{r}^{\ve_{1},\d})-B(X_{r}^{\ve_{2},\d}))_{H}\, dW_{r}\\
 & +C\int_{0}^{t}e^{-Kr}\|X_{r}^{\ve_{1},\d}-X_{r}^{\ve_{2},\d}\|_{H}^{2}\, dr\\
 & -K\int_{0}^{t}e^{-Kr}\|X_{r}^{\ve_{1},\d}-X_{r}^{\ve_{2},\d}\|_{H}^{2}\, dr.
\end{align*}
Using the Burkholder-Davis-Gundy inequality and \eqref{eq:strong_soln_visc} we obtain
\begin{align*}
\E\sup_{t\in[0,T]}e^{-Kt}\|X_{t}^{\ve_{1},\d}-X_{t}^{\ve_{2},\d}\|_{H}^{2}\le & 2\E\|x_{0}^{1}-x_{0}^{2}\|_{H}^{2}\\
 & +C(\ve_{1}+\ve_{2})(\E\|x_{0}^{1}\|_{H^{1}}^{2}+\E\|x_{0}^{2}\|_{H^{1}}^{2}+1),
\end{align*}
for $K>0$ large enough. Taking the limit $\d\to0$ yields (by step one)
\begin{align}
\E\sup_{t\in[0,T]}e^{-Kt}\|X_{t}^{\ve_{1}}-X_{t}^{\ve_{2}}\|_{H}^{2}\le & 2E\|x_{0}^{1}-x_{0}^{2}\|_{H}^{2}\label{eq:stability_ic_SFDE}\\
 & +C(\ve_{1}+\ve_{2})(\E\|x_{0}^{1}\|_{H^{1}}^{2}+\E\|x_{0}^{2}\|_{H^{1}}^{2}+1).\nonumber 
\end{align}
Hence, there is an $\{\mcF_{t}\}$-adapted process $X\in L^{2}(\O;C([0,T];H))$ with $X_{0}=x_{0}$ such that
\[
\E\sup_{t\in[0,T]}\|X_{t}^{\ve}-X_{t}\|_{H}^{2}\to0\quad\text{for }\ve\to0.
\]
\emph{Step 4: Approximating the initial condition ($n\to\infty$).} Let $X^{\ve,\d,n}$ be the unique strong solution \eqref{eq:full_approx_SVI_constr} and $X^{\d,n},X^{n}$ be the limits constructed in the last two steps. Taking $\ve\to0$ in \eqref{eq:stability_ic_SFDE} yields
\begin{align*}
\E\sup_{t\in[0,T]}e^{-Kt}\|X_{t}^{n}-X_{t}^{m}\|_{H}^{2}\le & 2\E\|x_{0}^{n}-x_{0}^{m}\|_{H}^{2}.
\end{align*}
Thus, there is an $\{\mcF_{t}\}$-adapted process $X\in L^{2}(\O;C([0,T];H))$ with $X_{0}=x_{0}$ such that 
\[
\E\sup_{t\in[0,T]}\|X_{t}^{n}-X_{t}\|_{H}^{2}\to0\quad\text{for }n\to\infty.
\]

\emph{Step 5: Energy inequality. }Itô's formula implies
\begin{align*}
\E e^{-tC}\|X_{t}^{\ve,\d,n}\|_{H}^{2}\le & \E\|x_{0}^{n}\|_{H}^{2}+2\E\int_{0}^{t}e^{-rC}(\ve\D X_{r}^{\ve,\d,n}+\div\phi^{\d}(\nabla X_{r}^{\ve,\d,n}),X_{r}^{\ve,\d,n})_{H}\, dr\\
 & +\E\int_{0}^{t}e^{-rC}\|B(X_{r}^{\ve,\d,n})\|_{L_{2}}^{2}\, dr-C\E\int_{0}^{t}e^{-rC}\norm{X_{r}^{\ve,\d,n}}_{H}^{2}\, dr.
\end{align*}
Since
\begin{align*}
 & (\ve\D X_{r}^{\ve,\d,n}+\div\phi^{\d}(\nabla X_{r}^{\ve,\d,n}),X_{r}^{\ve,\d,n})_{H}\\
 & =(\ve\D X_{r}^{\ve,\d,n},X_{r}^{\ve,\d,n})_{H}-(\phi^{\d}(\nabla X_{r}^{\ve,\d,n}),\nabla X_{r}^{\ve,\d,n})_{H}\\
 & \le\int_{\mcO}\psi^{\d}(\nabla X_{r}^{\ve,\d,n})d\xi\\
 & \le\vp(X_{r}^{\ve,\d,n})+C\d(\|X_{r}^{\ve,\d,n}\|_{H^{1}}^{2}+1)
\end{align*}
and 
\[
\|B(X_{r}^{\ve,\d,n})\|_{L_{2}}^{2}\lesssim1+\|X_{r}^{\ve,\d,n}\|_{H}^{2},
\]
choosing $C$ large enough yields
\begin{align*}
 & \E e^{-tC}\|X_{t}^{\ve,\d,n}\|_{H}^{2}+2\E\int_{0}^{t}e^{-rC}\vp(X_{r}^{\ve,\d,n})\, dr\le C(\E\|x_{0}^{n}\|_{H}^{2}+1)+C\d(\|X_{r}^{\ve,\d,n}\|_{H^{1}}^{2}+1).
\end{align*}
Using lower-semicontinuity of $\vp$ and \eqref{eq:strong_soln_visc} we may take the limit $\d\to0$ and, subsequently, the limits $\ve\to0$, $n\to\infty$ to obtain \eqref{eq:SVI_energy}.

\emph{Step 6: Variational inequality.}

Let now $F$, $G$, $Z$ be as in Definition \ref{def:SPDE_SVI} (with $H=L^{2}(\mcO)$ and $S=H^{1}(\mcO)$) and let $X^{\ve,\d,n}$ be the solution to \eqref{eq:full_approx_SVI_constr} with initial conditions $x_{0}^{n}\in L^{2}(\O,\mcF_{0};H^{1})$ satisfying $x_{0}^{n}\to x_{0}$ in $L^{2}(\O;H)$. Itô's formula implies
\begin{align*}
 & \E e^{-tK}\|X_{t}^{\ve,\d,n}-Z_{t}\|_{H}^{2}\\
 & =\E\|x_{0}^{n}-Z_{0}\|_{H}^{2}+2\E\int_{0}^{t}e^{-rK}(\ve\D X_{r}^{\ve,\d,n}+\div\phi^{\d}(\nabla X_{r}^{\ve,\d,n})-G_{r},X_{r}^{\ve,\d,n}-Z_{r})_{H}\, dr\\
 & +\E\int_{0}^{t}e^{-rK}\|B(X_{r}^{\ve,\d,n})-F_{r}\|_{L_{2}}^{2}\, dr\\
 & -K\E\int_{0}^{t}e^{-rK}\norm{X_{r}^{\ve,\d,n}-Z_{r}}_{H}^{2}\, dr.
\end{align*}
Due to \eqref{eq:yosida_convergence-2} we have
\begin{align*}
|\vp(v)-\vp^{\delta}(v)| & \le C\d(1+\vp(v))\quad\forall v\in H^{1}(\mcO)
\end{align*}
and thus (using convexity of $\psi^{\d}$ and \eqref{eq:MY-ineq})
\begin{align*}
(\div\phi^{\d}(\nabla X_{r}^{\ve,\d,n}),X_{r}^{\ve,\d,n}-Z_{r})_{H}\le & \vp^{\d}(Z_{r})-\vp^{\d}(X_{r}^{\ve,\d,n})\\
\le & \vp(Z_{r})-\vp(X_{r}^{\ve,\d,n})+C\d(1+\vp(X_{r}^{\ve,\d,n})),
\end{align*}
$dr\otimes\P$-a.e.. Moreover,
\begin{align*}
(\ve\D X_{r}^{\ve,\d,n},X_{r}^{\ve,\d,n}-Z_{r})_{H} & \le\ve\|\D X_{r}^{\ve,\d,n}\|_{H}\|X_{r}^{\ve,\d,n}-Z_{r}\|_{H}\\
 & \le\ve^{\frac{4}{3}}\|\D X_{r}^{\ve,\d,n}\|_{H}^{2}+\ve^{\frac{2}{3}}\|X_{r}^{\ve,\d,n}-Z_{r}\|_{H}^{2}
\end{align*}
$dr\otimes\P$-a.e.. Since
\begin{align*}
\|B(X_{r}^{\ve,\d,n})-F_{r}\|_{L_{2}}^{2} & \le\|B(X_{r}^{\ve,\d,n})-B(Z_{r})\|_{L_{2}}^{2}+\|B(Z_{r})-F_{r}\|_{L_{2}}^{2}\\
 & \le2L^{2}\|X_{r}^{\ve,\d,n}-Z_{r}\|_{H}^{2}+2\|B(Z_{r})-F_{r}\|_{L_{2}}^{2},
\end{align*}
we conclude that
\begin{align*}
 & \E e^{-tK}\|X_{t}^{\ve,\d,n}-Z_{t}\|_{H}^{2}+2\E\int_{0}^{t}e^{-rK}\vp(X_{r}^{\ve,\d,n})\, dr\\
\le & \E\|x_{0}^{n}-Z_{0}\|_{H}^{2}+2\E\int_{0}^{t}e^{-rK}\vp(Z_{r})\, dr+C\d\E\int_{0}^{t}e^{-rK}(1+\vp(X_{r}^{\ve,\d,n}))\, dr\\
 & -2\E\int_{0}^{t}e^{-rK}(G_{r},X_{r}^{\ve,\d,n}-Z_{r})_{H}\, dr+2\E\int_{0}^{t}e^{-rK}\|B(Z_{r})-F_{r}\|_{L_{2}}^{2}\, dr\\
 & +2\E\int_{0}^{t}e^{-rK}\left(\ve^{\frac{4}{3}}\|\D X_{r}^{\ve,\d,n}\|_{H}^{2}+\ve^{\frac{2}{3}}\|X_{r}^{\ve,\d,n}-Z_{r}\|_{H}^{2}\right)\, dr.
\end{align*}
Note that $\vp(v)\lesssim\|v\|_{H^{1}}^{2}+1$ for $v\in H^{1}$. Using \eqref{eq:strong_soln_visc} we may now first let $\d\to0$, then $\ve\to0$ and then $n\to\infty$ to obtain \eqref{eq:SVI} by lower-semicontinuity of $\vp$ on $H$.

\emph{Step 7: Uniqueness.}

Let $X$ be a continuous SVI solution to \eqref{eq:singular_p_laplace-1} and let $Y^{\ve,\d,n}$ be the (strong) solution to \eqref{eq:full_approx_SVI_constr} with initial condition $y_{0}^{n}\in L^{2}(\O;H^{1})$ satisfying $y_{0}^{n}\to y_{0}$ in $L^{2}(\O;H)$. Then \eqref{eq:SVI} with $Z=Y^{\ve,n}$, $F=B(Z)$ and $G=\ve\D Y^{\ve,n}+\div\phi^{\d}(\nabla Y^{\ve,n})$ yields 
\begin{align*}
 & \E e^{-tK}\|X_{t}-Y_{t}^{\ve,n}\|_{H}^{2}+2\E\int_{0}^{t}e^{-rK}\vp(X_{r})\, dr\\
\le & \E\|x_{0}-y_{0}^{n}\|_{H}^{2}+2\E\int_{0}^{t}e^{-rK}\vp(Y_{r}^{\ve,n})\, dr\\
 & -2\E\int_{0}^{t}e^{-rK}(\ve\D Y_{r}^{\ve,n}+\div\phi^{\d}(\nabla Y_{r}^{\ve,n}),X_{r}-Y_{r}^{\ve,n})_{H}\, dr,
\end{align*}
for a.e.~$t\in[0,T]$. By \eqref{eq:yosida_convergence-2}, for all $x\in H^{1}$ we have
\[
-(\div\phi^{\d}(\nabla Y^{\ve,\d,n}),x-Y^{\ve,\d,n})_{H}+\vp(Y^{\ve,\d,n})\le\vp(x)+C\d(1+\vp(Y^{\ve,\d,n}))\quad dr\otimes\P-\text{a.e.}.
\]
Since $\vp$ is the lower-semicontinuous hull of $\vp$ restricted to $H^{1}$, for a.e.~$(t,\o)\in[0,T]\times\O$, we can choose a sequence $x^{m}\in H^{1}$ such that $x^{m}\to X_{t}(\o)$ and $\vp(x^{m})\to\vp(X_{t}(\o))$. Hence,
\[
-(\div\phi^{\d}(\nabla Y^{\ve,\d,n}),X-Y^{\ve,\d,n})_{H}+\vp(Y^{\ve,\d,n})\le\vp(X)+C\d(1+\vp(Y^{\ve,\d,n}))\quad dr\otimes\P-\text{a.e.}.
\]
Thus,
\begin{align*}
\E e^{-tK}\|X_{t}-Y_{t}^{\ve,\d,n}\|_{H}^{2} & \le\E\|x_{0}-y_{0}^{n}\|_{H}^{2}+C\d\E\int_{0}^{t}e^{-rK}(1+\vp(Y_{r}^{\ve,\d,n}))\, dr\\
 & +2\E\int_{0}^{t}e^{-rK}\left(\ve^{\frac{4}{3}}\|\D Y_{r}^{\ve,\d,n}\|_{H}^{2}+\ve^{\frac{2}{3}}\|X_{r}-Y_{r}^{\ve,\d,n}\|_{H}^{2}\right)\, dr.
\end{align*}
Taking $\d\to0$ then $\ve\to0$ (using \eqref{eq:strong_soln_visc}) and then $n\to\infty$ yields
\begin{align*}
\E\|X_{t}-Y_{t}\|_{H}^{2}\le & e^{tK}\E\|x_{0}-y_{0}\|_{H}^{2},
\end{align*}
for a.e.~$t\in[0,T]$.
\end{proof}

\section{SVI approach to stochastic nonlocal $p$-Laplace equations\label{sec:nonlocal_plp}}

In this section we derive an SVI formulation for stochastic singular nonlocal $p$-Laplace equations with homogeneous Neumann boundary condition of the type
\begin{align}
dX_{t} & \in\left(\int_{\mcO}J(\cdot-\xi)|X_{t}(\xi)-X_{t}(\cdot)|^{p-2}(X_{t}(\xi)-X_{t}(\cdot))\, d\xi\right)dt+B(X_{t})\, dW_{t}\label{eq:nonlocal_plp-1}\\
X_{0} & =x_{0}\in L^{2}(\O,\mcF_{0};L^{2}(\mcO)),\nonumber 
\end{align}
where $p\in[1,2)$, $W$ is a cylindrical Wiener process on some separable Hilbert space $U$, $B:L^{2}(\mcO)\to L_{2}(U,L^{2}(\mcO))$ is Lipschitz continuous and $\mcO$ is a bounded, smooth domain in $\R^{d}$. The kernel $J:\R^{d}\to\R$ is supposed to be a nonnegative, continuous, radial function with compact support, $J(0)>0$ and $\int_{\R^{d}}J(z)\, dz=1$. In particular, we include the multivalued, limiting case $p=1$, for which we set $|r|^{-1}r=\sgn(r)$ to be the maximal monotone extension of the sign function.

In the following we develop an SVI approach to \eqref{eq:nonlocal_plp-1}, thus providing a unified treatment for SPDE of the type \eqref{eq:nonlocal_plp-1} including the multivalued case $p=1$. We let $S=H:=L^{2}(\mcO)$ and define
\[
\vp(u):=\frac{1}{2p}\int_{\mcO}\int_{\mcO}J\left(\z-\xi\right)\lrabs{u(\xi)-u(\z)}^{p}\, d\z\, d\xi,\quad u\in H.
\]
It is easy to see that $\vp$ defines a continuous, convex function on $H$ with subdifferential, if $p>1$, 
\[
A(u):=-\partial\vp(u)=\int_{\mcO}J(\cdot-\xi)|u(\xi)-u(\cdot)|^{p-2}(u(\xi)-u(\cdot))\, d\xi
\]
and, if $p=1$, 
\begin{align*}
A(u):= & -\partial\vp(u)\\
= & \Big\{\int_{\mcO}J(\cdot-\xi)\eta(\xi,\cdot)\, d\xi:\,\|\eta\|_{L^{\infty}}\le1,\,\eta(\xi,\z)=-\eta(\z,\xi)\text{ and }\\
 & \hskip20ptJ(\z-\xi)\eta(\xi,\z)\in J(\z-\xi)\sgn(u(\xi)-u(\z))\text{ for a.e. }(\xi,\z)\in\mcO\times\mcO\Big\},
\end{align*}
for $u\in H$. Hence, we may rewrite \eqref{eq:nonlocal_plp-1} as 
\begin{align*}
dX_{t} & \in-\partial\vp(X_{t})\, dt+B(X_{t})\, dW_{t}\\
X_{0} & =x_{0}.
\end{align*}
There exists an SVI solution to \eqref{eq:nonlocal_plp-1} by Proposition \ref{prop:ex-svi-sln}. Furthermore,
\begin{thm}
\label{thm:SVI-nonlocal}Let $x_{0}\in L^{2}(\Omega,\mcF_{0};H).$ Suppose that \eqref{eq:lip-lip} is satisfied. Then there is a unique continuous SVI solution $X$ to \eqref{eq:nonlocal_plp-1} in the sense of Definition \ref{def:SPDE_SVI}. For two SVI solutions $X$, $Y$ with initial conditions $x_{0},y_{0}\in L^{2}(\Omega;H)$ we have
\begin{equation}
\esssup_{t\in[0,T]}\E\|X_{t}-Y_{t}\|_{H}^{2}\lesssim\E\|x_{0}-y_{0}\|_{H}^{2}.\label{eq:uniqueness-contraction}
\end{equation}
 \end{thm}
\begin{proof}
We start by proving the existence of continuous SVI solutions to \eqref{eq:nonlocal_plp-1}. We recall that Proposition \ref{prop:ex-svi-sln} implies the existence of SVI solutions to \eqref{eq:nonlocal_plp-1} based on the Moreau-Yosida approximation of $\vp$. In order to prove uniqueness of (continuous) SVI solutions to \eqref{eq:nonlocal_plp-1} we need to consider an alternative approximation $\vp^{\d}$. Indeed, it turns out that in order to prove uniqueness of SVI solutions it is essential that the approximations satisfy $\vp^{\d}(v)\ge\vp(v)+\mathrm{Err}(v)$ for some well-controlled error term $\mathrm{Err}$. For the Moreau-Yosida approximation we rather have $\vp^{\d}\le\vp$ and no lower bound on $\vp^{\d}$ is known in general.

\emph{Step 1: Strong approximating SPDE.} We consider non-singular approximations of the nonlinearity $\vp$: Let $\psi,\psi^{\d},\phi^{\d},R_{\d}$ be as in Appendix \ref{sec:Moreau-Yosida}. We then consider 
\begin{align}
\vp^{\d}(u) & :=\frac{1}{2}\int_{\mcO}\int_{\mcO}J\left(\z-\xi\right)\psi^{\d}(u(\xi)-u(\z))\, d\xi\, d\z\label{eq:nonlocal_nonlin_approx}\\
A^{\d}(u):=-\partial\vp^{\d}(u) & =\int_{\mcO}J(\cdot-\xi)\phi^{\d}(u(\xi)-u(\cdot))\, d\xi,\quad u\in H.\nonumber 
\end{align}
and, as a strong approximation, the non-singular, non-degenerate SPDE:
\begin{align}
dX_{t}^{\d} & =-\partial\vp^{\d}(X_{t}^{\d})dt+B(X_{t}^{\d})dW_{t},\label{eq:nonlocal_approx}\\
X_{0}^{\d} & =x_{0}.\nonumber 
\end{align}
By \cite{PR07} there is a unique variational solution to \eqref{eq:nonlocal_approx} constructed along the trivial Gelfand triple $V=H\subseteq V^{*}$ and with $\a=2$. We verify, keeping in mind that $V=H=V^{\ast}$:
\begin{enumerate}
\item [(H1)] Hemi-continuity: $A:V\to V^{*}$ is continuous.
\item [(H2)] Monotonicity (compare with \cite[Lemma 6.5]{AMRT11}):
\begin{align*}
 & 2{}_{V^{*}}\<A^{\d}(u)-A^{\d}(v),u-v\>_{V}\\
= & \int_{\mcO}\int_{\mcO}J(\z-\xi)\phi^{\d}(u(\xi)-u(\z))((u-v)(\z)-(u-v)(\xi))\, d\xi\, d\z\\
 & -\int_{\mcO}\int_{\mcO}J(\z-\xi)\phi^{\d}(v(\xi)-v(\z))((u-v)(\z)-(u-v)(\xi))\, d\xi\, d\z\\
= & -\int_{\mcO}\int_{\mcO}J(\z-\xi)\left(\phi^{\d}(u(\xi)-u(\z))-\phi^{\d}(v(\xi)-v(\z))\right)\\
 & \hskip38pt(u(\xi)-u(\z)-(v(\xi)-v(\z)))\, d\xi\, d\z\\
\le & 0.
\end{align*}

\item [(H3)] Coercivity:
\begin{align*}
2{}_{V^{*}}\<A^{\d}(u),u\>_{V} & =-\int_{\mcO}\int_{\mcO}J(\z-\xi)\phi^{\d}(u(\xi)-u(\z))(u(\xi)-u(\z))\, d\xi\, d\z\\
 & \le\|u\|_{H}^{2}-\|u\|_{H}^{2}.
\end{align*}

\item [(H4)] Growth: Using Hölder's inequality
\begin{align*}
|{}_{V^{*}}\<A^{\d}(v),u\>_{V}|\le & \frac{1}{2}\int_{\mcO}\int_{\mcO}J^{\frac{1}{2}}(\z-\xi)|\phi^{\d}|(v(\xi)-v(\z))J^{\frac{1}{2}}(\z-\xi)|u(\xi)-u(\z)|\, d\xi\, d\z\\
\le & \frac{1}{2}\left(\int_{\mcO}\int_{\mcO}J(\z-\xi)|\phi^{\d}|^{2}(v(\xi)-v(\z))\, d\z\, d\xi\right)^{\frac{1}{2}}\\
 & \left(\int_{\mcO}\int_{\mcO}J(\z-\xi)|u(\xi)-u(\z)|^{2}\, d\xi\, d\z\right)^{\frac{1}{2}}\\
\lesssim & \left(\int_{\mcO}\int_{\mcO}J(\z-\xi)|\phi^{\d}|^{2}(v(\xi)-v(\z))\, d\z\, d\xi\right)^{\frac{1}{2}}\|u\|_{V}.
\end{align*}
By \eqref{eq:phi-delta-bound} we have $|\phi^{\d}|^{2}(r)\le C(1+|r|^{2})$ and thus
\begin{align*}
\|A(v)\|_{V^{*}} & \le C\left(1+\|v\|_{V}\right).
\end{align*}

\end{enumerate}
Using \cite[Theorem 4.2.4]{PR07} there is a unique variational solution $X^{\d}$ to \eqref{eq:nonlocal_approx} and
\begin{equation}
\E\sup_{t\in[0,T]}\|X_{t}^{\d}\|_{H}^{2}\le C<\infty,\label{eq:nonlocal_bound}
\end{equation}
for some constant $C>0$ independent of $\d>0$. Since $A^{\d}:H\to H$ is Lipschitz continuous $X^{\d}$ is a strong solution to \eqref{eq:nonlocal_approx}.

\emph{Step 2: Convergence for $\d\to0$. }For two solutions $X^{\d_{1}},X^{\d_{2}}$ to \eqref{eq:nonlocal_approx} with initial condition $x_{0}\in L^{2}(\O;H)$ we have by Itô's formula 
\begin{align*}
e^{-Kt}\|X_{t}^{\d_{1}}-X_{t}^{\d_{2}}\|_{H}^{2}= & 2\int_{0}^{t}e^{-Kr}(-\partial\vp^{\d_{1}}(X_{r}^{\d_{1}})+\partial\vp^{\d_{2}}(X_{r}^{\d_{2}}),X_{r}^{\d_{1}}-X_{r}^{\d_{2}})_{H}\, dr\\
 & +2\int_{0}^{t}e^{-Kr}(X_{r}^{\d_{1}}-X_{r}^{\d_{2}},B(X_{r}^{\d_{1}})-B(X_{r}^{\d_{2}}))_{H}\, dW_{r}\\
 & +\int_{0}^{t}e^{-Kr}\|B(X_{r}^{\d_{1}})-B(X_{r}^{\d_{2}})\|_{L_{2}}^{2}\, dr\\
 & -K\int_{0}^{t}e^{-Kr}\|X_{r}^{\d_{1}}-X_{r}^{\d_{2}}\|_{H}^{2}\, dr.
\end{align*}
We observe that
\begin{align*}
 & -(\partial\vp^{\d_{1}}(u)-\partial\vp^{\d_{2}}(v),u-v)_{H}\\
 & =-\int_{\mcO}\int_{\mcO}J(\z-\xi)\left(\phi^{\d_{1}}(u(\xi)-u(\z))-\phi^{\d_{2}}(v(\xi)-v(\z))\right)\\
 & \hskip49pt(u(\xi)-u(\z)-(v(\xi)-v(\z)))\, d\xi\, d\z
\end{align*}
and due to \eqref{eq:monotone_Y_bound} we obtain
\begin{align*}
 & -(\partial\vp^{\d_{1}}(u)-\partial\vp^{\d_{2}}(v),u-v)_{H}\\
 & \le C(\d_{1}+\d_{2})\int_{\mcO}\int_{\mcO}J(\z-\xi)\left(1+|u(\xi)-u(\z)|^{2}+|v(\xi)-v(\z)|^{2}\right)\, d\xi\, d\z\\
 & \le C(\d_{1}+\d_{2})\left(1+\|u\|_{H}^{2}+\|v\|_{H}^{2}\right).
\end{align*}
In conclusion,
\begin{align*}
e^{-Kt}\|X_{t}^{\d_{1}}-X_{t}^{\d_{2}}\|_{H}^{2}= & C(\d_{1}+\d_{2})\int_{0}^{t}e^{-Kr}\left(1+\|X_{r}^{\d_{1}}\|_{H}^{2}+\|X_{r}^{\d_{2}}\|_{H}^{2}\right)\, dr\\
 & +2\int_{0}^{t}e^{-Kr}(X_{r}^{\d_{1}}-X_{r}^{\d_{2}},B(X_{r}^{\d_{1}})-B(X_{r}^{\d_{2}}))_{H}\, dW_{r}\\
 & +\int_{0}^{t}e^{-Kr}\|B(X_{r}^{\d_{1}})-B(X_{r}^{\d_{2}})\|_{L_{2}}^{2}\, dr\\
 & -K\int_{0}^{t}e^{-Kr}\|X_{r}^{\d_{1}}-X_{r}^{\d_{2}}\|_{H}^{2}\, dr.
\end{align*}
Using the Burkholder-Davis-Gundy inequality and \eqref{eq:nonlocal_bound}, we obtain
\begin{align}
\E\sup_{t\in[0,T]}e^{-Kt}\|X_{t}^{\d_{1}}-X_{t}^{\d_{2}}\|_{H}^{2}\le & C(\d_{1}+\d_{2})(\E\|x_{0}\|_{H}^{2}+1),\label{eq:stability_ic_SFDE-2-1}
\end{align}
for $K>0$ large enough. Hence, we obtain the existence of a sequence of $\{\Fcal_{t}\}$-adapted, time-continuous processes $X^{\d}\in L^{2}(\O;C([0,T];H))$ with $X_{0}^{\d}=x_{0}$ and an $\{\Fcal_{t}\}$-adapted process $X\in L^{2}(\O;C([0,T];H))$ with $X_{0}=x_{0}$ such that
\[
\E\sup_{t\in[0,T]}\|X_{t}^{\d}-X_{t}\|_{H}^{2}\to0\quad\text{for }\d\to0.
\]

\emph{Step 3:}\textit{\emph{ }}\textit{Energy inequality. }An application of Itô's formula and a standard localization argument yield
\[
\E\|X_{t}^{\d}\|_{H}^{2}=\E\|x_{0}\|_{H}^{2}-2\E\int_{0}^{t}(\partial\vp^{\d}(X_{r}^{\d}),X_{r}^{\d})_{H}dr+\E\int_{0}^{t}\|B(X_{r}^{\d})\|_{L_{2}(U,H)}^{2}dr.
\]
By the definition of the subdifferential we have
\begin{align*}
(-\partial\vp^{\d}(X^{\d}),X^{\d})_{H} & =(\partial\vp^{\d}(X^{\d}),0-X^{\d})_{H}\le-\vp^{\d}(X^{\d})\quad dt\otimes\P-\text{a.s.}
\end{align*}
and by Lipschitz continuity of $B$
\[
\|B(X_{r}^{\d})\|_{L_{2}(U,H)}^{2}\le C(1+\|X_{r}^{\d}\|_{H}^{2}).
\]
Hence, using Gronwall's Lemma yields
\[
\E e^{-Ct}\|X_{t}^{\d}\|_{H}^{2}+\E\int_{0}^{t}e^{-Cr}\vp^{\d}(X_{r}^{\d})dr\lesssim\E\|x_{0}\|_{H}^{2}+1.
\]
Due to \eqref{eq:yosida_convergence-2} we thus obtain
\[
\E\|X_{t}^{\d}\|_{H}^{2}+\E\int_{0}^{t}\vp(X_{r}^{\d})dr\lesssim\E\|x_{0}\|_{H}^{2}+1+\d\E\int_{0}^{t}\|X_{r}^{\d}\|_{H}^{2}dr.
\]
Taking the limit $\d\to0$ and using lower semicontinuity of $v\mapsto\E\int_{0}^{t}\vp(v)dr$ on $L^{2}([0,T]\times\O;H)$ yields Definition \ref{def:SPDE_SVI}, (i).

\emph{Step 4: Variational inequality. }It remains to prove that the time-continuous process $X$ solves the SVI. Since $X^{\d}$ is a strong solution to \eqref{eq:nonlocal_approx}, by Remark \ref{rmk:varn_sol} for each $(Z,F,G,Z_{0})$ as in Definition \ref{def:SPDE_SVI} we have that 
\begin{align}
 & \E e^{-Ct}\|X_{t}^{\d}-Z_{t}\|_{H}^{2}+2\E\int_{0}^{t}e^{-Cr}\vp^{\d}(X_{r}^{\d})\, dr\nonumber \\
 & \le\E\|x_{0}-Z_{0}\|_{H}^{2}+2\E\int_{0}^{t}e^{-Cr}\vp^{\d}(Z_{r})\, dr-2\E\int_{0}^{t}e^{-Cr}(G_{r},X_{r}^{\d}-Z_{r})_{H}\, dr\label{eq:SVI-1-1-1}\\
 & +2\E\int_{0}^{t}e^{-Cr}\|F_{r}-B(Z_{r})\|_{L_{2}(U,H)}^{2}\, dr\quad\forall t\in[0,T].\nonumber 
\end{align}
By \eqref{eq:yosida_convergence-2} we have 
\begin{align*}
|\vp^{\d}(Z_{r})-\vp(Z_{r})| & \lesssim\d(1+\vp(Z_{r}))\\
 & \lesssim\d(1+\|Z_{r}\|_{H}^{2}).
\end{align*}
Mosco convergence of $\vp^{\delta}\to\vp$ can easily be verified using Fatou's lemma and Lebesgue's dominated convergence and the fact that $\psi^{\delta}$ converges pointwise and Mosco to $\psi.$ Hence, by Mosco convergence of integral functionals (see Appendix \ref{sec:Random-Mosco-convergence}), taking the limit in \eqref{eq:SVI-1-1-1} implies that $X$ is a continuous SVI solution to \eqref{eq:nonlocal_plp-1}.

\emph{Step 5: Uniqueness.} Let $X$ be an SVI solution to \eqref{eq:nonlocal_plp-1} and let $\{Y^{\d}\}$ be the (strong) solution to \eqref{eq:nonlocal_approx} with initial condition $y_{0}\in L^{2}(\O;H)$. Then \eqref{eq:test-eq} with $Z=Y^{\d}$, $F=B(Z)$ and $G=-\partial\vp^{\d}(Y^{\d})$ yield 
\begin{align*}
 & \E e^{-tC}\|X_{t}-Y_{t}^{\d}\|_{H}^{2}+2\E\int_{0}^{t}e^{-rC}\vp(X_{r})\, dr\\
\le & \E\|x_{0}-y_{0}\|_{H}^{2}+2\E\int_{0}^{t}e^{-rC}\vp(Y_{r}^{\d})\, dr\\
 & +2\E\int_{0}^{t}e^{-rC}(\partial\vp^{\d}(Y_{r}^{\d}),X_{r}-Y_{r}^{\d})_{H}\, dr\quad\text{for a.e. }t\ge0.
\end{align*}
By the subgradient property and \eqref{eq:MY-ineq},
\[
(\partial\vp^{\d}(Y_{r}^{\d}),X_{r}-Y_{r}^{\d})_{H}+\vp^{\d}(Y_{r}^{\d})\le\vp^{\d}(X_{r})\le\vp(X_{r})\quad dr\otimes\P-\text{a.e.}.
\]
Moreover, due to \eqref{eq:yosida_convergence-2} we have 
\begin{align*}
|\vp^{\d}(Y_{r}^{\d})-\vp(Y_{r}^{\d})| & \lesssim\d(1+\vp(Y_{r}^{\d}))\\
 & \lesssim\d(1+\|Y_{r}^{\d}\|_{H}^{2}).
\end{align*}
Thus,
\begin{align*}
\E\|X_{t}-Y_{t}^{\d}\|_{H}^{2} & \le\E e^{tC}\|x_{0}-y_{0}\|_{H}^{2}+\d(1+\E\|Y_{r}^{\d}\|_{H}^{2})\quad\text{for a.e. }t\ge0.
\end{align*}
Since by step two we have $Y^{\d}\to Y$ in $C([0,T];L^{2}(\O;H))$ we may take the limit $\d\to0$, which by weak lower semicontinuity of the norm concludes the proof.
\end{proof}

\section{Convergence of stochastic nonlocal to local $p$-Laplace equations\label{sec:to_local}}

In this section, we investigate the convergence of the solutions to the stochastic nonlocal $p$-Laplace equation to solutions of the stochastic (local) $p$-Laplace equation under appropriate rescaling of the kernel $J$.

More precisely, let $\Ocal\subset\R^{d}$ be a bounded, convex, smooth domain and let $J:\R^{d}\to\R$ be a nonnegative continuous radial function with compact support, $J(0)>0$, $\int_{\R^{d}}J(z)\, dz=1$ and $J(x)\ge J(y)$ for all $|x|\le|y|$.

For $p\in[1,2)$, $\eps>0$, we then define the rescaled functionals 
\[
\vp^{\ve}(u):=\frac{C_{J,p}}{2p\eps^{d}}\int_{\Ocal}\int_{\Ocal}J\left(\frac{\xi-\zeta}{\eps}\right)\lrabs{\frac{u(\zeta)-u(\xi)}{\eps}}^{p}\, d\zeta d\xi,
\]
for $u\in L^{p}(\mcO)$, where 
\[
C_{J,p}^{-1}:=\frac{1}{2}\int_{\R^{d}}J(z)|z_{d}|^{p}\, dz.
\]

Furthermore, for $p\in(1,2),$ we set 
\[
\vp(u):=\begin{cases}
\frac{1}{p}\int_{\Ocal}\abs{\nabla u}^{p}\, d\xi, & \quad\text{if}\quad u\in W^{1,p}(\Ocal),\\
+\infty, & \quad\text{if}\quad u\in L^{p}(\Ocal)\setminus W^{1,p}(\Ocal),
\end{cases}
\]
whereas, for $p=1$, we set
\[
\vp(u):=\begin{cases}
\|u\|_{TV}, & \quad\text{if}\quad u\in BV(\Ocal),\\
+\infty, & \quad\text{if}\quad u\in L^{1}(\Ocal)\setminus BV(\Ocal),
\end{cases}
\]
By Theorem \ref{thm:SVI-nonlocal} for each $\ve>0$, there is a unique time-continuous SVI solution $X^{\ve}$ to the stochastic nonlocal $p$-Laplace equation 
\begin{align}
dX_{t}^{\ve} & \in-\partial_{L^{2}}\varphi_{\ve}(X_{t}^{\ve})\, dt+B(X_{t}^{\ve})dW_{t},\label{eq:nonlocal_plp}\\
X_{0}^{\ve} & =x\nonumber 
\end{align}
and by Theorem \ref{thm:SVI} there is a unique time-continuous SVI solution to the stochastic (local) $p$-Laplace equation
\begin{align}
dX_{t} & \in-\partial_{L^{2}}\varphi(X_{t})\, dt+B(X_{t})dW_{t},\label{eq:local_plp}\\
X_{0} & =x,\nonumber 
\end{align}
where $\partial_{L^{2}}\varphi$ denotes the $L^{2}$ subgradient of $\vp$ restricted to $L^{2}.$
\begin{thm}
\label{thm:weak-convergence-thm} Let $x_{0}\in L^{2}(\Omega,\mcF_{0};H)$ and let $X^{\ve}$, $X$ be the time-continuous SVI solution to \eqref{eq:nonlocal_plp}, \eqref{eq:local_plp} respectively. Then
\[
X^{\ve}\rightharpoonup X\quad\text{in }L^{2}([0,T]\times\O;H).
\]
\end{thm}
\begin{proof}
We shall verify the conditions of Proposition \ref{prop:random_Mosco}, (ii), which will conclude the proof by an application of Theorem \ref{thm:SVI_stability}. Hence, we need to show that \eqref{eq:bound_for_vp-2} is satisfied. To do so, we first note that
\begin{align}
C_{J,p}^{-1} & =\frac{1}{2}\int_{\R^{d}}J(|z|)|z\cdot e_{d}|^{p}dz\label{eq:radial_profile_calculation}\\
 & =\frac{1}{2}\int_{\R^{d}}J(|z|)|z|^{p}\left|\frac{z}{|z|}\cdot e_{d}\right|^{p}dz\nonumber \\
 & =\frac{1}{2}\int_{\R_{+}}\int_{S^{d-1}}J(r)r^{p+d-1}|\s\cdot e_{d}|^{p}d\s dr\nonumber \\
 & =\frac{K_{p,d}}{2}\int_{\R_{+}}J(r)r^{p+d-1}dr,\nonumber 
\end{align}
where
\[
K_{p,d}:=\int_{S^{d-1}}|\s\cdot e_{d}|^{p}d\s.
\]
Hence,
\[
\frac{C_{J,p}K_{p,d}}{2}\int_{\R_{+}}J(r)r^{p+d-1}dr=1.
\]
Thus, by \cite[Proposition IX.3]{B83}, for each $u\in W^{1,p}(\Ocal)=\dom(\vp)$, if $p\in(1,2)$, 
\begin{align*}
\vp^{\ve}(u) & =\frac{C_{J,p}}{2p\eps^{d}}\int_{\Ocal}\int_{\Ocal}J\left(\frac{\xi-\zeta}{\eps}\right)\lrabs{\frac{u(\zeta)-u(\xi)}{\eps}}^{p}\, d\zeta d\xi\\
 & =\frac{C_{J,p}}{2p\eps^{p}}\int_{\R^{d}}J\left(z\right)\int_{\Ocal}\lrabs{\bar{u}(\xi+\ve z)-u(\xi)}^{p}\, d\xi dz\\
 & \le\frac{C_{J,p}}{2p\eps^{p}}\int_{\R^{d}}J\left(z\right)\int_{\Ocal}\lrabs{\nabla u(\xi)}^{p}\, d\xi|\ve z|^{p}dz\\
 & =\frac{C_{J,p}}{2p}\int_{\R^{d}}J\left(z\right)|z|^{p}dz\int_{\Ocal}\lrabs{\nabla u(\xi)}^{p}\, d\xi\\
 & =d\s_{d}\frac{C_{J,p}}{2p}\int_{\R_{+}}J\left(r\right)r^{p+d-1}dr\int_{\Ocal}\lrabs{\nabla u(\xi)}^{p}\, d\xi\\
 & =\frac{d\s_{d}}{K_{p,d}}\frac{1}{p}\int_{\Ocal}\lrabs{\nabla u(\xi)}^{p}\, d\xi\\
 & =C\vp(u),
\end{align*}
for each $u\in BV(\Ocal)$, if $p=1$, resp., by \cite[eqs. (14)--(16)]{Brez2002}
\begin{align*}
\vp^{\ve}(u) & =\frac{C_{J,1}}{2\eps^{d}}\int_{\Ocal}\int_{\Ocal}J\left(\frac{\xi-\zeta}{\eps}\right)\lrabs{\frac{u(\zeta)-u(\xi)}{\eps}}\, d\zeta d\xi\\
 & =\frac{C_{J,1}}{2\eps}\int_{\R^{d}}J\left(z\right)\int_{\Ocal}\lrabs{\bar{u}(\xi+\ve z)-u(\xi)}\, d\xi dz\\
 & \le\frac{C_{J,1}}{2\eps}\int_{\R^{d}}J\left(z\right)|\ve z|dz\,|Du|(\Ocal)\\
 & =\frac{C_{J,1}}{2}\int_{\R^{d}}J\left(z\right)|z|dz\,|Du|(\Ocal)\\
 & =d\s_{d}\frac{C_{J,1}}{2}\int_{\R_{+}}J\left(r\right)r^{d}dr\,|Du|(\Ocal)\\
 & =\frac{d\s_{d}}{K_{1,d}}\|u\|_{TV}\\
 & =C\vp(u),
\end{align*}
where we have denoted the total variation of the vector measure $Du$ by $|Du|$, that is, $|Du|(\Ocal)=\|u\|_{TV}$. By Proposition \ref{prop:neumann-convergence} below we can apply Theorem \ref{thm:SVI_stability} to conclude the proof.\end{proof}
\begin{prop}
\label{prop:neumann-convergence}Let $\eps_{n}\searrow0$ as $n\to\infty$. Then
\begin{enumerate}
\item For each sequence $u^{\ve_{n}}\rightharpoonup u$ weakly in $L^{p}(\Ocal)$ as $n\to\infty$, we have that 
\[
\liminf_{n\to\infty}\vp^{\ve_{n}}(u^{\ve_{n}})\ge\vp(u).
\]

\item For each $u\in W^{1,p}(\Ocal)$ (if $p\in(1,2)$), for each $u\in BV(\Ocal)$, resp. (if $p=1$), it holds that 
\[
\lim_{n\to\infty}\vp^{\ve_{n}}(u)=\vp(u).
\]

\end{enumerate}
In particular, $\vp^{\ve}\to\vp$ in Mosco sense in $L^{2}$.\end{prop}
\begin{proof}
\textit{(i):}\textbf{ }For simplicity set $v^{n}:=u^{\ve_{n}}$ and $\vp^{n}:=\vp^{\ve_{n}}$. Clearly, $\sup_{n\in\N}\norm{v^{n}}_{L^{p}(\Ocal)}\le C$ for some constant $C>0$. Without loss of generality we may assume 
\[
\liminf_{n\to\infty}\vp^{n}(v^{n})<+\infty
\]
Suppose therefore, after extracting a subsequence if necessary (denoted by $v^{n}$, too), that 
\[
\liminf_{n\to\infty}\vp^{n}(v^{n})=\lim_{n\to\infty}\vp^{n}(v^{n}).
\]
In particular, 
\[
\sup_{n\in\N}\vp^{n}(v^{n})\le C
\]
for some constant $C>0$. We get that 
\[
\frac{C_{J,p}}{2}\int_{\Ocal}\int_{\Ocal}\eps_{n}^{-d}J\left(\frac{\xi-\zeta}{\eps_{n}}\right)\lrabs{v^{n}(\zeta)-v^{n}(\xi)}^{p}\, d\zeta\, d\xi\le Cp\eps_{n}^{p}.
\]
\emph{Case: $p\in(1,2)$}

By \cite[Theorem 6.11]{AVMRT10} it follows that $u\in W^{1,p}(\Ocal)$ and 
\[
\left(\frac{C_{J,p}}{2}J(z)\right)^{1/p}1_{\Ocal}(\xi+\eps_{n}z)\frac{\bar{v}^{n}(\xi+\eps_{n}z)-v^{n}(\xi)}{\eps_{n}}\rightharpoonup\left(\frac{C_{J,p}}{2}J(z)\right)^{1/p}z\cdot\nabla u(\xi)
\]
weakly in $L^{p}(\Ocal)\times L^{p}(\R^{d})$. 

Note that by variable substitution,
\begin{align*}
 & \frac{C_{J,p}}{2p}\int_{\Ocal}\int_{\R^{d}}J(z)1_{\Ocal}(\xi+\eps_{n}z)\lrabs{\frac{\bar{v}^{n}(\xi+\eps_{n}z)-v^{n}(\xi)}{\eps_{n}}}^{p}\, dz\, d\xi\\
 & =\frac{C_{J,p}}{2p\eps_{n}^{d}}\int_{\Ocal}\int_{\Ocal}J\left(\frac{\xi-\zeta}{\eps_{n}}\right)\lrabs{\frac{v^{n}(\zeta)-v^{n}(\xi)}{\eps_{n}}}^{p}\, d\zeta\, d\xi\\
 & =\vp^{n}(v^{n}).
\end{align*}
 Let $\eta(\xi,z)\in L^{p/(p-1)}(\Ocal)\times L^{p/(p-1)}(\R^{d})$ be a test-function. Then by Young's inequality
\begin{align*}
 & \int_{\Ocal}\int_{\R^{d}}\left(\frac{C_{J,p}}{2}J(z)\right)^{1/p}1_{\Ocal}(\xi+\eps_{n}z)\frac{\bar{v}^{n}(\xi+\eps_{n}z)-v^{n}(\xi)}{\eps_{n}}\,\eta(\xi,z)\, dz\, d\xi\\
 & \le\frac{C_{J,p}}{2p}\int_{\Ocal}\int_{\R^{d}}J(z)1_{\Ocal}(\xi+\eps_{n}z)\lrabs{\frac{\bar{v}^{n}(\xi+\eps_{n}z)-v^{n}(\xi)}{\eps_{n}}}^{p}\, dz\, d\xi\\
 & +\frac{p-1}{p}\int_{\Ocal}\int_{\R^{d}}|\eta(\xi,z)|^{p/(p-1)}\, dz\, d\xi.
\end{align*}
Upon taking the limit $n\to\infty$ we obtain that
\begin{align*}
 & \int_{\Ocal}\int_{\R^{d}}\left(\frac{C_{J,p}}{2}J(z)\right)^{1/p}z\cdot\nabla u(\xi)\,\eta(\xi,z)\, dz\, d\xi\\
 & \le\liminf_{n\to\infty}\vp^{n}(v^{n})+\frac{p-1}{p}\int_{\Ocal}\int_{\R^{d}}\abs{\eta(\xi,z)}^{p/(p-1)}\, dz\, d\xi.
\end{align*}
Choosing 
\[
\eta(\xi,z):=\left(\frac{C_{J,p}}{2}J(z)\right)^{(p-1)/p}\abs{z\cdot\nabla u(\xi)}^{p-2}z\cdot\nabla u(\xi),
\]
which is in $L^{p/(p-1)}(\Ocal)\times L^{p/(p-1)}(\R^{d})$ (recall that $J$ has compact support), yields
\begin{align*}
 & \int_{\Ocal}\int_{\R^{d}}\frac{C_{J,p}}{2}J(z)\abs{z\cdot\nabla u(\xi)}^{p}\, dz\, d\xi\\
 & \le\liminf_{n\to\infty}\vp^{n}(v^{n})+\frac{p-1}{p}\int_{\Ocal}\int_{\R^{d}}\frac{C_{J,p}}{2}J(z)\abs{z\cdot\nabla u(\xi)}^{p}\, dz\, d\xi.
\end{align*}
Hence, 
\[
\frac{1}{p}\int_{\Ocal}\int_{\R^{d}}\frac{C_{J,p}}{2}J(z)\abs{z\cdot\nabla u(\xi)}^{p}\, dz\, d\xi\le\liminf_{n\to\infty}\vp^{n}(v^{n})
\]
By \cite[Lemma 6.16]{AVMRT10}, 
\begin{align*}
 & \int_{\Ocal}\int_{\R^{d}}\frac{C_{J,p}}{2}J(z)\abs{z\cdot\nabla u(\xi)}^{p}\, dz\, d\xi\\
 & =\int_{\Ocal}\int_{\R^{d}}\sum_{i=1}^{d}\frac{C_{J,p}}{2}J(z)\abs{z\cdot\nabla u(\xi)}^{p-2}z\cdot\nabla u(\xi)z_{j}\partial_{j}u(\xi)\, dz\, d\xi\\
 & =\int_{\Ocal}\abs{\nabla u(\xi)}^{p}\, d\xi.
\end{align*}
Hence, we have proved that 
\[
\vp(u)\le\liminf_{n\to\infty}\vp^{n}(v^{n}).
\]
Since the above arguments work for any subsequence of $u^{\ve_{n}}$ this concludes the proof for $p\in(1,2)$.

\emph{Case: $p=1$}

By \cite[Theorem 6.11]{AVMRT10}, it follows that $u\in BV(\Ocal)$ and
\[
\frac{C_{J,1}}{2}J(z)1_{\Ocal}(\xi+\eps_{n}z)\frac{\bar{v}^{n}(\xi+\eps_{n}z)-v^{n}(\xi)}{\eps_{n}}\rightharpoonup\sum_{i=1}^{d}\frac{C_{J,p}}{2}J(z)z_{i}D_{i}u
\]
weakly in the sense of measures.

Let $\eta(\xi,z)\in C_{b}(\bar{\Ocal}\times\R^{d})$ be a test function. Then clearly
\begin{align*}
 & \int_{\Ocal}\int_{\R^{d}}\frac{C_{J,1}}{2}J(z)1_{\Ocal}(\xi+\eps_{n}z)\frac{\bar{v}^{n}(\xi+\eps_{n}z)-v^{n}(\xi)}{\eps_{n}}\,\eta(\xi,z)\, dz\, d\xi\\
 & \le\|\eta\|_{\infty}\frac{C_{J,1}}{2}\int_{\Ocal}\int_{\R^{d}}J(z)1_{\Ocal}(\xi+\eps_{n}z)\lrabs{\frac{\bar{v}^{n}(\xi+\eps_{n}z)-v^{n}(\xi)}{\eps_{n}}}\, dz\, d\xi\\
 & =\|\eta\|_{\infty}\vp^{n}(v^{n}).
\end{align*}
Upon taking the limit $n\to\infty$ we obtain that,
\begin{align*}
 & \sum_{i=1}^{d}\int_{\Ocal}\int_{\R^{d}}\frac{C_{J,1}}{2}J(z)z_{i}\eta(\xi,z)\, dz\, d[D_{i}u]\le\|\eta\|_{\infty}\liminf_{n\to\infty}\vp^{n}(v^{n}).
\end{align*}
Taking the supremum over all test functions of the form $\eta$ such that $\|\eta\|_{\infty}\le1$ yields by \cite[Proposition 1.47]{AFP00},
\[
\frac{C_{J,1}}{2}|\mu|(\Ocal\times\R^{d})\le\liminf_{n\to\infty}\vp^{n}(v^{n}),
\]
where $|\mu|(\Ocal\times\R^{d})$ denotes the total variation of the signed Radon measure $\mu(d\xi,dz)=\sum_{i=1}^{d}J(z)z_{i}dz\, d[D_{i}u]$.

Since by \cite[proof of Theorem 7.10, p. 174]{AVMRT10},
\[
\frac{C_{J,1}}{2}|\mu|(\Ocal\times\R^{d})=|Du|(\Ocal)=\|u\|_{TV},
\]
we get that
\[
\vp(u)\le\liminf_{n\to\infty}\vp^{n}(v_{n}).
\]
Since the arguments work for any subsequence this concludes the proof.

\textit{(ii):}\textbf{ }Taking \eqref{eq:radial_profile_calculation} into account, recall that 
\[
\frac{C_{J,p}K_{p,d}}{2}\int_{\R_{+}}J(r)r^{p+d-1}dr=1.
\]
\emph{Case: $p\in(1,2)$}

By\cite[Theorem 2', Corollary 4, D]{Brez2002} applied with $\g(r)=\frac{C_{J,p}K_{p,d}}{2}J(r)$, for $u\in W^{1,p}(\Ocal)$, we have that
\begin{align*}
 & \frac{C_{J,p}K_{p,d}}{2}\lim_{\ve\to0}\frac{1}{\ve^{d+p}}\int_{\mcO}\int_{\mcO}J\left(\frac{|\z-\xi|}{\ve}\right)|u(\z)-u(\xi)|^{p}d\z d\xi\\
 & =\lim_{\ve\to0}\frac{1}{\ve^{d+p}}\int_{\mcO}\int_{\mcO}\g\left(\frac{|\z-\xi|}{\ve}\right)|u(\z)-u(\xi)|^{p}d\z d\xi\\
 & =K_{p,d}\int_{\Ocal}\abs{\nabla u}^{p}\, d\xi.
\end{align*}
Hence, for $u\in W^{1,p}(\Ocal)$,
\begin{align*}
\lim_{\ve\to0}\vp^{\ve}(u) & =\lim_{\ve\to0}\frac{C_{J,p}}{2p\eps^{d}}\int_{\Ocal}\int_{\Ocal}J\left(\frac{\xi-\zeta}{\eps}\right)\lrabs{\frac{u(\zeta)-u(\xi)}{\eps}}^{p}\, d\zeta d\xi\\
 & =\frac{1}{p}\int_{\Ocal}\abs{\nabla u}^{p}\, d\xi\\
 & =\vp(u).
\end{align*}
\emph{Case: $p=1$}

Again, by \cite[Theorem 2', Corollary 4, D]{Brez2002}, for $u\in BV(\Ocal)$, we get that
\begin{align*}
 & \frac{C_{J,1}K_{1,d}}{2}\lim_{\ve\to0}\frac{1}{\ve^{d+1}}\int_{\mcO}\int_{\mcO}J\left(\frac{|\z-\xi|}{\ve}\right)|u(\z)-u(\xi)|d\z d\xi\\
 & =\lim_{\ve\to0}\frac{1}{\ve^{d+1}}\int_{\mcO}\int_{\mcO}\g\left(\frac{|\z-\xi|}{\ve}\right)|u(\z)-u(\xi)|d\z d\xi\\
 & =K_{1,d}|Du|(\Ocal).
\end{align*}
Hence, for $u\in BV(\Ocal)$,
\begin{align*}
\lim_{\ve\to0}\vp^{\ve}(u) & =\lim_{\ve\to0}\frac{C_{J,1}}{2\eps^{d}}\int_{\Ocal}\int_{\Ocal}J\left(\frac{\xi-\zeta}{\eps}\right)\lrabs{\frac{u(\zeta)-u(\xi)}{\eps}}\, d\zeta d\xi\\
 & =|Du|(\Ocal)\\
 & =\vp(u).
\end{align*}

\end{proof}

\section{Trotter type results \label{sec:trotter}}

\subsection{Stochastic $p$-Laplace equations\label{sec:trotter_plp}}

We consider stochastic singular $p$-Laplace evolution equations with zero Neumann boundary conditions
\begin{align}
dX_{t} & \in\div\phi\left(\nabla X_{t}\right)\, dt+B(X_{t})\, dW_{t},\nonumber \\
\phi(\nabla X_{t})\cdot\nu & \ni0\quad\mathrm{on}\;\partial\Ocal,\; t>0,\label{eq:singular_p_laplace}\\
X_{0} & =x_{0}\nonumber 
\end{align}
on bounded, convex, smooth domains $\mcO\subseteq\R^{d}$ , where $\nu$ denotes the outer normal on $\partial\Ocal$ and $\phi=\partial\psi$ is given as the subdifferential of a convex function $\psi:\R^{d}\to\R_{+}$ satisfying 
\begin{equation}
\psi(z)=\td\psi(|z|)\label{eq:plp-psi}
\end{equation}
for some convex, continuous, non-decreasing function $\td\psi$ and
\begin{equation}
\psi(z)\le C(1+|z|^{2})\quad\forall z\in\R^{d}.\label{eq:plp-growth}
\end{equation}
In particular, we are interested in singular $p$-Laplace equations, that is, $\phi(z)=|z|^{p-2}z$ with $p\in[1,2)$. Note that this includes the stochastic total variation flow for $p=1$. In the following let $H=L^{2}$, $S=H^{1}$ and $B,W$ be as in Section \ref{sec:generalities_SVI}. Further, let
\begin{equation}
\tilde{\vp}(u):=\begin{cases}
\int_{\Ocal}\psi(\nabla u)\, d\xi & \text{if }u\in H^{1}\\
+\infty & \text{if }u\in L^{2}\setminus H^{1}
\end{cases}\label{eq:trotter_plp_vp}
\end{equation}
and let $\vp$ be the l.s.c. hull of $\tilde{\vp}$ on $L^{2}$. We may then write \eqref{eq:singular_p_laplace} in its relaxed form
\begin{align}
dX_{t} & \in-\partial\vp(X_{t})dt+B(X_{t})dW_{t}.\label{eq:SFDE-1-1}
\end{align}

From \cite[Example 7.9]{GT11} we recall that there is a unique (limit) solution to \eqref{eq:singular_p_laplace}, which by a slight modification%
\footnote{In \cite[Appendix C]{GT11} SVI solutions are defined for the special choice $F=B(Z)$ in Definition \ref{def:SPDE_SVI} (cf.~also Remark \ref{rem:SVI_notions}). However, it is easy to see that the same arguments as in \cite[Appendix C]{GT11} can also be employed for general $F$, thus leading to an SVI solution in the sense of Definition \ref{def:SPDE_SVI}.%
} of \cite[Appendix C]{GT11} is also an SVI solution to \eqref{eq:singular_p_laplace}. 
\begin{thm}
\label{thm:trotter_plp}Let $\psi^{n}$ be a sequence of convex functions satisfying \eqref{eq:plp-psi} and \eqref{eq:plp-growth} with a constant $C>0$ independent of $n$. Suppose that $\psi^{n}\to\psi:=\frac{1}{p}|\cdot|^{p}$ in Mosco sense and 
\[
\limsup_{n\to\infty}\psi^{n}(z)\le\psi(z)\quad\forall z\in\R^{d}.
\]
Let $X^{n}$ be the unique (limit) solutions to \eqref{eq:singular_p_laplace} with $\psi$ replaced by $\psi^{n}$ and $X$ the unique SVI solution to \eqref{eq:singular_p_laplace} with $\psi$ as above. Then 
\[
X^{n}\rightharpoonup X\quad\text{in }L^{2}([0,T]\times\O;H)
\]
for $n\to\infty$.\end{thm}
\begin{proof}
Let $\tilde{\vp}^{n},\vp^{n}$ as in \eqref{eq:trotter_plp_vp} with $\psi$ replaced by $\psi^{n}$. By Proposition \ref{prop:Mosco_plp_trotter} below we know that $\vp^{n}\to\vp$ in Mosco sense and $\limsup_{n\to\infty}\vp^{n}(u)\le\vp(u)$. Hence, the proof follows from Proposition \ref{prop:random_Mosco} and Theorem \ref{thm:SVI_stability}.\end{proof}
\begin{prop}
\label{prop:Mosco_plp_trotter}Let $\psi^{n},\psi:\R^{d}\to\R_{+}$ be l.s.c., convex functions satisfying $\psi^{n}(0)=\psi(0)=0$ and \eqref{eq:plp-growth} for some constant $C>0$ independent of $n$. Suppose that $\psi^{n}\to\psi$ in Mosco sense and 
\begin{equation}
\limsup_{n\to\infty}\psi^{n}(z)\le\psi(z)\quad\forall z\in\R^{d}.\label{eq:ptw_ineq-2}
\end{equation}
Let $\tilde{\vp}_{n}$, $\tilde{\vp}$ be as in \eqref{eq:trotter_plp_vp} with l.s.c.~hull on $L^{2}$ denoted by $\vp^{n}$, $\vp$ respectively. Then $\vp^{n}\to\vp$ in Mosco sense in $L^{2}$ and 
\begin{equation}
\limsup_{n\to\infty}\vp^{n}(u)\le\vp(u)\quad\forall u\in L^{2}.\label{eq:ptw_ineq-1}
\end{equation}
\end{prop}
\begin{proof}
In the following let $R_{1}^{n}$ denote the resolvent corresponding to $\partial\vp^{n}$, that is, for $f\in L^{2}$, $z=R_{1}^{n}f$ is the unique solution to 
\begin{equation}
z+\partial\vp^{n}(z)\ni f.\label{eq:res_eqn-1}
\end{equation}
Equivalently,
\[
(f,v-z)_{L^{2}}+\vp^{n}(z)+\frac{{1}}{2}\|z\|_{L^{2}}^{2}\le\vp^{n}(v)+\frac{{1}}{2}\|v\|_{L^{2}}^{2}\quad\forall v\in L^{2}.
\]
Analogously let $R_{1}$ be the resolvent of $\partial\vp$. We prove convergence of the resolvents $R_{1}^{n}f$ to $R_{1}f$ for all $f\in L^{2}$, which by \cite[Theorem 3.66]{A84} implies the desired Mosco convergence of $\vp^{n}$ to $\vp$ . In order to prove convergence of the resolvents, in a first step we need to establish an $H^{1}$ bound.

\textit{Step 1: }In this step we prove that
\begin{equation}
\|R_{1}^{n}f\|_{H^{1}}\le\|f\|_{H^{1}}.\label{eq:l2_res_bound-1}
\end{equation}

In the following we consider $n\in\N$ fixed and suppress it in the notation. Let $f\in H^{1}$. We proceed by considering a non-degenerate, non-singular approximation of $\vp$, that is, we define 
\[
\vp^{\lambda}(u):=\begin{cases}
\int_{\Ocal}\psi^{\lambda}(\nabla u)+\frac{{\lambda}}{2}|\nabla u|^{2}\, d\xi & u\in H^{1}\\
+\infty & u\in L^{2}\setminus H^{1},
\end{cases}
\]
where $\psi^{\lambda}$ denotes the Moreau-Yosida approximation of $\psi$. Then $\vp^{\lambda}$ is easily seen to be l.s.c. on $L^{2}$. Moreover,
\[
\mcD(\partial\vp^{\lambda})=H_{N}^{2}:=\{v\in H^{2}\,:\,\nabla v\cdot\nu=0\text{ on }\partial\mcO\}
\]
with
\begin{equation}
-\partial\vp^{\lambda}(u)=\div\phi^{\lambda}(\nabla u)+\l\Delta u\quad\forall u\in H_{N}^{2},\label{eq:approx_char-1}
\end{equation}
where $\phi^{\lambda}:=(\psi^{\lambda})^{\prime}$. We now consider the resolvent equation corresponding to $\vp^{\l}$, that is, 
\[
z^{\lambda}-\div\phi^{\lambda}(\nabla z^{\lambda})-\lambda\Delta z^{\lambda}=z^{\lambda}+\partial\vp^{\l}(z^{\lambda})=f.
\]
In particular, we have $z^{\l}\in\mcD(\partial\vp)=H_{N}^{2}$ and $\div\phi^{\lambda}(\nabla z^{\lambda})+\lambda\Delta z^{\lambda}\in L^{2}$. Multiplying with $-\D z^{\lambda}$ and integrating yields
\begin{align*}
\|z^{\lambda}\|_{\dot{H}^{1}}^{2}+(\div\phi^{\lambda}(\nabla z^{\lambda})+\lambda\Delta z^{\lambda},\Delta z^{\lambda})_{L^{2}} & \le\|f\|_{\dot{H}^{1}}\|z^{\lambda}\|_{\dot{H}^{1}}.
\end{align*}
As in \eqref{eq:liu-ref} we observe that 
\begin{equation}
(\div\phi^{\lambda}(\nabla z^{\lambda})+\lambda\Delta z^{\lambda},\Delta z^{\lambda})_{L^{2}}\le0\label{eq:liu-ref-2}
\end{equation}
and hence
\begin{equation}
\|z^{\lambda}\|_{H^{1}}\le\|f\|_{H^{1}}.\label{eq:approx_l2-1}
\end{equation}
Let $z^{\ast}$ be a weak accumulation point of $z^{\lambda}$ in $H^{1}$. By Mosco convergence of integral functionals (cf.~Appendix \ref{sec:Random-Mosco-convergence}), we have $\vp^{\lambda}\to\vp$ in Mosco sense in $H^{1}$. Hence, for $v\in H^{1}$, we can pass to the limit in the resolvent equation for $\vp^{\l}$, that is, in 
\[
(f,v-z^{\lambda})_{L^{2}}+\vp^{\lambda}(z^{\lambda})+\frac{{1}}{2}\|z^{\lambda}\|_{L^{2}}^{2}\le\vp^{\lambda}(v)+\frac{{1}}{2}\|v\|_{L^{2}}^{2}
\]
and we get that by weak lower semi-continuity of the norm and Lebesgue's dominated convergence theorem 
\[
(f,v-z^{\ast})_{L^{2}}+\vp(z^{*})+\frac{{1}}{2}\|z^{*}\|_{L^{2}}^{2}\le\vp(v)+\frac{{1}}{2}\|v\|_{L^{2}}^{2},
\]
for all $v\in H^{1}$. Since $\vp$ is the l.s.c. hull of $\td\vp$ , for each $v\in L^{2}$ there is a sequence $v_{n}\in H^{1}$ such that $v_{n}\to v$ in $L^{2}$ and $\limsup_{n\to\infty}\td\vp(v_{n})\le\vp(v)$. Hence, for each $v\in L^{2}$ we obtain that
\[
(f,v-z^{\ast})_{L^{2}}+\vp(z^{*})+\frac{{1}}{2}\|z^{*}\|_{L^{2}}^{2}\le\vp(v)+\frac{{1}}{2}\|v\|_{L^{2}}^{2},
\]
and, hence, $z^{\ast}$ is the resolvent $R_{1}f$ of $\partial\vp$, that is, 
\[
z^{\ast}+\partial\vp(z^{\ast})\ni f.
\]
By \eqref{eq:approx_l2-1} we conclude
\[
\|z^{\ast}\|_{H^{1}}\le\|f\|_{H^{1}}.
\]

\textit{Step 2: }Now, let $f\in H^{1}$ and consider the sequence of resolvent $z_{n}=R_{1}^{n}f$, that is, 
\[
z_{n}+\partial\vp^{n}(z_{n})\ni f
\]
By step one we have that 
\[
\|z_{n}\|_{H^{1}}\le\|f\|_{H^{1}}.
\]
Let $z^{\ast}$ be a weak accumulation point of $z_{n}$ in $H^{1}$. By Mosco convergence of integral functionals (cf.~Appendix \ref{sec:Random-Mosco-convergence}) we have $\vp^{n}\to\vp$ in Mosco sense on $H^{1}$. Moreover, by reverse Fatou inequality, $\limsup_{n}\tilde{\vp}_{n}(v)\le\tilde{\vp}(v)$ pointwise in $H^{1}$. Hence, for $v\in H^{1}$, we can pass to the limit in
\begin{equation}
(f,v-z_{n})_{L^{2}}+\vp_{n}(z_{n})+\frac{{1}}{2}\|z_{n}\|_{L^{2}}^{2}\le\vp_{n}(v)+\frac{{1}}{2}\|v\|_{L^{2}}^{2},\label{eq:var_ineq_phi_n-1}
\end{equation}
to obtain
\[
(f,v-z^{\ast})_{L^{2}}+\vp(z^{\ast})+\frac{{1}}{2}\|z^{\ast}\|_{L^{2}}^{2}\le\vp(v)+\frac{{1}}{2}\|v\|_{L^{2}}^{2}.
\]
Since $\vp_{}$ is the l.s.c. hull of $\td\vp_{}$ , for each $v\in L^{2}$ there is a sequence $v_{m}\in H^{1}$ such that $v_{n}\to v$ in $L^{2}$ and $\limsup_{n\to\infty}\td\vp(v_{m})\le\vp(v)$. Therefore, we obtain 
\[
(f,v-z^{\ast})_{L^{2}}+\vp(z^{\ast})+\frac{{1}}{2}\|z^{\ast}\|_{L^{2}}^{2}\le\vp(v)+\frac{{1}}{2}\|v\|_{L^{2}}^{2},
\]
for all $v\in L^{2}$ or equivalently
\[
z^{\ast}+\partial\vp(z^{\ast})\ni f.
\]
Setting $v=z^{\ast}$ in \eqref{eq:var_ineq_phi_n-1}, yields $\limsup_{n}\|z_{n}\|_{L^{2}}\le\|z^{\ast}\|_{L^{2}}$ and hence by weak lower semi-continuity of the norm $\|z_{n}\|_{L^{2}}\to\|z^{\ast}\|_{L^{2}}$. By the Kadets-Klee property of Hilbert spaces, we deduce strong convergence $z_{n}\to z^{\ast}$ in $L^{2}$. In conclusion, for $f\in H^{1}$ we have shown
\[
R_{1}^{n}f\to R_{1}f\quad\text{in }L^{2}
\]
for $n\to\infty$. By density of the embedding $H^{1}\subset L^{2}$ and \cite[Theorem 3.62]{A84} this convergence holds for all $f\in L^{2}$. By \cite[Theorem 3.66]{A84} this implies Mosco convergence of $\vp^{n}$ to $\vp$. 

The inequality \eqref{eq:ptw_ineq-1} follows using \eqref{eq:ptw_ineq-2} and the reverse Fatou inequality.
\end{proof}
Specific examples of approximations $\psi^{n}$ of $\psi$ in Theorem \ref{thm:trotter_FDE} one may consider (note that pointwise convergence of $\psi^{n}$ to $\psi$ on $\R^{d}$ implies Mosco convergence, cf.~\cite[Example 5.13]{DM93})
\begin{example}
~
\begin{enumerate}
\item Convergence of powers: Let $p_{n}\in[1,2)$ be a sequence such that $p_{n}\to p_{0}$ for some $p_{0}\in[1,2)$ and set $\psi^{n}(\cdot):=\frac{1}{p_{n}}|\cdot|^{p_{n}}$.
\item Vanishing viscosity: Let $\psi^{n}(z)=\frac{1}{2n}|z|^{2}+\psi(z)$. 
\item Yosida-approximation: Let $\psi^{n}$ be the Moreau-Yosida approximation of $\psi(\cdot):=\frac{1}{p}|\cdot|^{p}$ for $p\in[1,2)$. 
\end{enumerate}
\end{example}

\subsection{Stochastic fast diffusion equations\label{sec:FDE-1}}

We consider stochastic generalized fast diffusion equations of the type
\begin{align}
dX_{t} & \in\D\phi(X_{t})dt+B(X_{t})dW_{t},\label{eq:SFDE}\\
X_{0} & =x_{0}\nonumber 
\end{align}
on bounded, smooth domains $\mcO\subseteq\R^{d}$ with zero Dirichlet boundary conditions, where $\phi=\partial\psi$ is given as the subdifferential of an even, convex, continuous function $\psi:\R\to\R_{+}$ satisfying
\begin{equation}
\psi(r)\le C(1+|r|^{2})\quad\forall r\in\R.\label{eq:FDE_growth_bound}
\end{equation}
In particular, we are interested in fast diffusion equations, i.e. $\psi(r)=\frac{1}{m+1}|r|^{m+1}$, $m\in[0,1]$. Note that this includes the multivalued case $m=0$. In this section we consider the stability of solutions to \eqref{eq:SFDE} with respect to $\phi$.

Let
\begin{equation}
\tilde{\vp}(u):=\begin{cases}
\int_{\Ocal}\psi(u)\, d\xi & \text{if }u\in L^{2}\\
+\infty & \text{if }u\in H^{-1}\setminus L^{2}
\end{cases}\label{eq:fde_functional}
\end{equation}
and $\vp$ be the l.s.c. hull of $\tilde{\vp}$ on $H^{-1}$. We may then write \eqref{eq:SFDE} in its relaxed form
\begin{align}
dX_{t} & \in-\partial\vp(X_{t})dt+B(X_{t})dW_{t}.\label{eq:SFDE-1}
\end{align}

In the following let $H=H^{-1},S=L^{2}(\mcO)$, where $H^{-1}$ is the dual of $H_{0}^{1}(\mcO)$. Further, let $B$, $W$ be as in Section \ref{sec:generalities_SVI}. By \cite[Example 7.3]{GT11}, for each $x_{0}\in L^{2}(\O,\mcF_{0};H)$ there is a unique (limit) solution to \eqref{eq:SFDE}. By a slight modification of \cite[Appendix C]{GT11} this solution is also a continuous SVI solution to \eqref{eq:SFDE}. We further note that by \cite{GR15} for $\psi(r)=\frac{1}{m+1}|r|^{m+1}$ with $m\in[0,1]$ there is a unique continuous SVI solution to \eqref{eq:SFDE}.
\begin{thm}
\label{thm:trotter_FDE}Let $\psi^{n}$ be a sequence of even, convex, continuous functions satisfying \eqref{eq:FDE_growth_bound} with a uniform $C>0$. Suppose that $\psi^{n}\to\psi(\cdot):=\frac{1}{m+1}|\cdot|^{m+1}$ for some $m\in[0,1]$ in Mosco sense and
\[
\limsup_{n\to\infty}\psi^{n}(r)\le\psi(r)\quad\forall r\in\R.
\]
Let $X^{n}$ be the unique (limit) solutions to \eqref{eq:SFDE} with $\psi$ replaced by \textup{$\psi^{n}$} and $X$ be the unique SVI solution to \eqref{eq:SFDE} with $\psi$ as above. Then 
\[
X^{n}\rightharpoonup X\quad\text{in }L^{2}([0,T]\times\O;H^{-1})
\]
for $n\to\infty$.\end{thm}
\begin{proof}
We aim to apply Proposition \ref{prop:random_Mosco} and Theorem \ref{thm:SVI_stability}. Let $\tilde{\vp}^{n}$, $\tilde{\vp}$ be as in \eqref{eq:fde_functional} with l.s.c. hull on $H^{-1}$ denoted by $\vp^{n}$, $\vp$ respectively. We need to prove that $\vp^{n}\to\vp$ in Mosco sense and $ $$\limsup_{n\to\infty}\vp^{n}(v)\le\vp(v)$ for all $v\in S$. Indeed, this holds by Proposition \ref{prop:Mosco_FDE_Trotter} below, which finishes the proof.\end{proof}
\begin{prop}
\label{prop:Mosco_FDE_Trotter}Let $\psi^{n},\psi:\R\to\R_{+}$ be l.s.c., convex functions satisfying $\psi^{n}(0)=\psi(0)=0$ and \eqref{eq:FDE_growth_bound} for some constant $C>0$ independent of $n$. Suppose that $\psi^{n}\to\psi$ in Mosco sense and 
\[
\limsup_{n\to\infty}\psi^{n}(r)\le\psi(r)\quad\forall r\in\R.
\]
Let $\tilde{\vp}^{n}$, $\tilde{\vp}$ be as in \eqref{eq:fde_functional} with l.s.c. hull on $H^{-1}$ denoted by $\vp^{n}$, $\vp$ respectively. Then $\vp^{n}\to\vp$ in Mosco sense in $H^{-1}$ and 
\begin{equation}
\limsup_{n\to\infty}\vp^{n}(u)\le\vp(u)\quad\forall u\in H^{-1}.\label{eq:ptw_ineq}
\end{equation}
\end{prop}
\begin{proof}
The proof follows along the same lines as Proposition \ref{prop:Mosco_plp_trotter}, replacing $L^{2}$ by $H^{-1}$, $H^{1}$ by $L^{2}$ and $H_{N}^{2}$ by $H_{0}^{1}$. The only difference appears in the derivation of the $L^{2}$ bound of $z^{\l}$, where instead of \eqref{eq:liu-ref-2} the elementary observation
\[
(\Delta(\phi^{\lambda}(z^{\lambda})+\lambda z^{\lambda}),z^{\lambda})_{L^{2}}\le0
\]
for $z^{\lambda}\in H_{0}^{1}$ and $\Delta(\phi^{\lambda}(z^{\lambda})+\lambda z^{\lambda})\in L^{2}$ is used. The details are left to the reader.
\end{proof}
As in Section \ref{sec:trotter_plp} As specific examples of approximations $\psi^{n}$ of $\psi$ in Theorem \ref{thm:trotter_FDE} one may consider
\begin{example}
~
\begin{enumerate}
\item Convergence of powers: Let $m_{n}\in[0,1]$ be a sequence such that $m_{n}\to m_{0}$ for some $m_{0}\in[0,1]$ and set $\psi^{n}(\cdot):=\frac{1}{m_{n}+1}|\cdot|^{m_{n}+1}$.
\item Vanishing viscosity: Let $\psi^{n}(r)=\frac{1}{2n}r^{2}+\psi(r)$. 
\item Yosida-approximation: Let $\psi^{n}$ be the Moreau-Yosida approximation of $\psi(\cdot):=\frac{1}{m+1}|\cdot|^{m+1}$. 
\end{enumerate}
\end{example}

\section{Homogenization\label{sec:homo}}

\subsection{Stochastic $p$-Laplace equations}

We consider the periodic homogenization problem for stochastic $p$-Laplace equations of the type 
\begin{align}
dX_{t} & =\div\left(a\left(\frac{\xi}{\ve}\right)|\nabla X_{t}|^{p-2}\nabla X_{t}\right)\, dt+B(X_{t})\, dW_{t},\nonumber \\
|\nabla X_{t}|^{p-2}\nabla X_{t}\cdot\nu & =0\quad\mathrm{on}\;\partial\Ocal,\; t>0,\label{eq:singular_p_laplace-homo}\\
X_{0} & =x_{0},\nonumber 
\end{align}
where $p\in(1,2)$ and $a\in L^{\infty}(\R^{d})$ is periodic on a cube $Y:=\prod_{i=1}^{d}[l_{i},r_{i})$, $l_{i}<r_{i}$, $1\le i\le d$ and $a\ge\rho>0$ for some constant $\rho>0.$ We note that the results from \cite{C11-1} applied to \eqref{eq:singular_p_laplace-homo} require, in addition, that $B=(-\D)^{-\s}$ for some $\s>0$ constant, $a\in C^{2}(Y)$ and $p=2$. We do not require these additional assumptions. We show that the solutions $X^{\ve}$ to \eqref{eq:singular_p_laplace-homo} converge to the homogenized limit
\begin{align}
dX_{t} & =M_{Y}(a)\div\left(|\nabla X_{t}|^{p-2}\nabla X_{t}\right)\, dt+B(X_{t})\, dW_{t},\nonumber \\
|\nabla X_{t}|^{p-2}\nabla X_{t}\cdot\nu & =0\quad\mathrm{on}\;\partial\Ocal,\; t>0,\label{eq:singular_p_laplace-homo-1}\\
X_{0} & =x_{0},\nonumber 
\end{align}
where
\[
M_{Y}(a):=\frac{1}{|Y|}\int_{Y}a(\xi)\, d\xi.
\]

For $u\in H:=L^{2}(\mcO)$ let
\[
\vp^{\ve}(u):=\begin{cases}
\frac{{1}}{p}\int_{\Ocal}a\left(\frac{{\xi}}{\eps}\right)|\nabla u(\xi)|^{p}\, d\xi & \quad\text{if }u\in W^{1,p}(\mcO)\\
+\infty & \quad\text{otherwise.}
\end{cases}
\]
and
\[
\vp(u):=\begin{cases}
\frac{M_{Y}(a)}{p}\int_{\Ocal}|\nabla u(\xi)|^{p}\, d\xi & \quad\text{if }u\in W^{1,p}(\mcO)\\
+\infty & \quad\text{otherwise.}
\end{cases}
\]
By \cite{RRW07}, for each $\ve>0$ there is a unique variational solution $X^{\ve}$ to \eqref{eq:singular_p_laplace-homo}, which as in Remark \ref{rmk:varn_sol} is easily seen to be a time-continuous SVI solution to \eqref{eq:singular_p_laplace-homo} with $H=L^{2}(\mcO)$, $S=H^{1}(\mcO)$. By Section \ref{sec:plp} there is a unique time-continuous SVI solution to \eqref{eq:singular_p_laplace-homo-1} with $H,S$ as before.
\begin{thm}
Let $x_{0}\in L^{2}(\O,\mcF_{0};H)$ and let $X^{\ve}$, $X$ be the solutions to \eqref{eq:singular_p_laplace-homo}, \eqref{eq:singular_p_laplace-homo-1} respectively. Then 
\[
X^{\ve}\rightharpoonup X\quad\text{in }L^{2}([0,T]\times\O;H)
\]
for $\ve\to0$.\end{thm}
\begin{proof}
The proof follows immediately from Theorem \ref{thm:mosco_homo_plp} below, Proposition \ref{prop:random_Mosco} and Theorem \ref{thm:SVI_stability}.\end{proof}
\begin{thm}
\label{thm:mosco_homo_plp}For $\eps\searrow0$ we have that $\vp^{\eps}\to\vp^{\operatorname{hom}}$ in Mosco sense in $L^{p}(\Ocal)$. Furthermore, for all $u\in L^{p}(\Ocal)$, we have that
\[
\limsup_{\ve\to0}\vp^{\ve}(u)\le\vp^{\operatorname{hom}}(u).
\]
\end{thm}
\begin{proof}
Let $\eps_{n}\to0$ and, by abuse of notation, set $\vp^{n}:=\vp^{\eps_{n}}$. Let $u^{n}\in L^{p}(\Ocal)$ such that $u^{n}\rightharpoonup u$ weakly in $L^{p}(\Ocal)$ for some $u\in L^{p}(\Ocal)$. W.l.o.g.
\[
\liminf_{n\to\infty}\vp^{n}(u^{n})<+\infty
\]
and for a non-relabeled subsequence $\vp^{n}(u^{n})<\infty$ and
\[
\lim_{n\to\infty}\vp^{n}(u^{n})=\liminf_{n\to\infty}\vp^{n}(u^{n})<+\infty.
\]
Hence,
\[
\rho\int_{\Ocal}|\nabla u^{n}(\xi)|^{p}\, d\xi\le\int_{\Ocal}a\left(\frac{\xi}{\ve_{n}}\right)|\nabla u^{n}(\xi)|^{p}\, d\xi\le C.
\]
Since $u^{n}$ is bounded in $W^{1,p}(\Ocal)$ a subsequence of $u^{n}$ converges weakly to some $u^{0}\in W^{1,p}(\Ocal)$. By the $L^{p}(\Ocal)$-weak convergence $u^{n}\rightharpoonup u$ we have $u^{0}=u\in W^{1,p}(\Ocal)$ and we conclude
\[
\int_{\Ocal}|\nabla u|^{p}\, d\xi\le C.
\]
By Young's inequality, for $\eta\in L^{q}(\Ocal;\R^{d})$, 
\[
\begin{split}\int_{\Ocal}\nabla u^{n}\eta\, d\xi & =\int_{\Ocal}\frac{\eta}{a\left(\frac{\xi}{\ve_{n}}\right)}\nabla u^{n}a\left(\frac{\xi}{\ve_{n}}\right)\, d\xi\\
 & \le\vp^{n}(u^{n})+\frac{1}{q}\int_{\Ocal}\left|\frac{\eta}{a\left(\frac{\xi}{\ve_{n}}\right)}\right|^{q}a\left(\frac{\xi}{\ve_{n}}\right)\, d\xi\\
 & =\vp^{n}(u^{n})+\frac{1}{q}\int_{\Ocal}|\eta|^{q}a^{1-q}\left(\frac{\xi}{\ve_{n}}\right)\, d\xi.
\end{split}
\]
Passing on to the limit, by \cite[Theorem 2.6]{CD99},
\[
\int_{\Ocal}\nabla u\eta\, d\xi\le\liminf_{n\to\infty}\vp^{n}(u^{n})+\frac{1}{q}\int_{\Ocal}M_{Y}(a^{1-q})|\eta|^{q}\, d\xi.
\]
Note that by Jensen's inequality, $M_{Y}(a^{1-q})\le M_{Y}(a)^{1-q}$. Hence, setting $\eta:=M_{Y}(a)|\nabla u|^{p-2}\nabla u\in L^{q}(\Ocal;\R^{d})$, yields,
\[
M_{Y}(a)\int_{\Ocal}|\nabla u|^{p}\, d\xi\le\liminf_{n\to\infty}\vp^{n}(u^{n})+\frac{{1}}{q}\int_{\Ocal}M_{Y}(a)|\nabla u|^{p}\, d\xi,
\]
and hence
\[
\vp^{\operatorname{hom}}(u)=\frac{M_{Y}(a)}{p}\int_{\Ocal}|\nabla u|^{p}\, d\xi\le\liminf_{n\to\infty}\vp^{n}(u^{n})
\]
and the first Mosco condition is proved. By \cite[Theorem 2.6]{CD99}, it is easy to see, that for all $u\in W^{1,p}(\Ocal)$,
\[
\lim_{n\to\infty}\frac{{1}}{p}\int_{\Ocal}a\left(\frac{\xi}{\ve_{n}}\right)|\nabla u|^{p}\, d\xi=\frac{M_{Y}(a)}{p}\int_{\Ocal}|\nabla u|^{p}\, d\xi.
\]
Hence,
\[
\limsup_{n\to\infty}\vp^{n}(u)\le\vp^{\operatorname{hom}}(u),
\]
for each $u\in L^{p}(\Ocal)$.
\end{proof}

\subsection{Stochastic fast diffusion equations\label{sec:homo_FDE}}

We consider the homogenization problem ($\ve\to0$) for stochastic fast diffusion equations of the type
\begin{align}
dX_{t} & =\D\left(a\left(\frac{{\xi}}{\eps}\right)X_{t}^{[m]}\right)dt+B(X_{t}^ {})dW_{t},\label{eq:SFDE-homo-1}\\
X_{0} & =x_{0},\nonumber 
\end{align}

with $m\in(0,1)$, on bounded, smooth domains $\mcO\subseteq\R^{d}$ with zero Dirichlet boundary conditions. Here, $a\in L^{\infty}(\R^{d})$ is periodic with respect to a cube $Y:=\prod_{i=1}^{d}[l_{i},r_{i})$, $l_{i}<r_{i}$, $1\le i\le d$ and bounded from below, i.e. $a\ge\rho>0$ for some constant $\rho>0.$ 

Note that in \cite{C11} the function $a$ was assumed to additionally satisfy: $a$ Lipschitz on $\bar{Y}$, $a\in C^{2}(Y)$ and $\D a\le0$. We do not require these additional assumptions. 

In this section we show that the solutions $X^{\ve}$ to \eqref{eq:SFDE-homo-1} converge to the unique continuous SVI solution to the homogenized limit
\begin{align}
dX_{t} & =M_{Y}(a)\D\left(X_{t}^{[m]}\right)dt+B(X_{t})dW_{t},\label{eq:SFDE-homo}\\
X_{0} & =x_{0}.\nonumber 
\end{align}
As in \cite{GR15} we define
\begin{align*}
L^{m+1}\cap H^{-1} & :=\left\{ v\in L^{m+1}\;:\;\int vhdx\le C\|h\|_{H_{0}^{1}},\ \forall h\in C_{c}^{1}(\mcO)\ \text{for some }C\ge0\right\} .
\end{align*}
For $u\in H^{-1}$ we set
\begin{align*}
\vp^{\ve}(u):= & \begin{cases}
\frac{{1}}{m+1}\int_{\Ocal}a\left(\frac{{\xi}}{\eps}\right)|u(\xi)|^{m+1}\, d\xi & \quad\text{if }u\in L^{m+1}\cap H^{-1}\\
\infty & \quad\text{otherwise. }
\end{cases}
\end{align*}
and 
\begin{align*}
\vp^{\operatorname{hom}}(u):= & \begin{cases}
\frac{M_{Y}(a)}{m+1}\int_{\Ocal}|u(\xi)|^{m+1}\, d\xi & \quad\text{if }u\in L^{m+1}\cap H^{-1}\\
\infty & \quad\text{otherwise. }
\end{cases}
\end{align*}

By \cite{RRW07} there is a unique variational solution $X^{\ve}$ to \eqref{eq:SFDE-homo-1} for each $\ve>0$. As in Remark \ref{rmk:varn_sol} it is easy to see that $X^{\ve}$ also is a continuous SVI solution to \eqref{eq:SFDE-homo-1} with $H=H^{-1}$, $S=L^{2}(\mcO)$. By \cite{GR15} there is a unique continuous SVI solution to \eqref{eq:SFDE-homo} with $H,S$ as before. 
\begin{thm}
Let $x_{0}\in L^{2}(\O,\mcF_{0};H)$ and let $X^{\ve}$, $X$ be the solutions to \eqref{eq:SFDE-homo-1}, \eqref{eq:SFDE-homo} respectively. Then
\[
X^{\ve}\rightharpoonup X\quad\text{in }L^{2}([0,T]\times\O;H)
\]
for $\ve\to0$.\end{thm}
\begin{proof}
Using Theorem \ref{thm:mosco_homo_FDE} below, the proof is a direct application of Proposition \ref{prop:random_Mosco} and Theorem \ref{thm:SVI_stability}. \end{proof}
\begin{thm}
\label{thm:mosco_homo_FDE}For $\eps\searrow0$ we have that $\vp^{\ve}\to\vp^{\operatorname{hom}}$ in Mosco sense in $H^{-1}$. Furthermore, for all $u\in H^{-1}$, we have that
\[
\limsup_{\ve\to0}\vp^{\ve}(u)\le\vp^{\operatorname{hom}}(u).
\]
\end{thm}
\begin{proof}
The proof proceeds similar to Theorem \ref{thm:mosco_homo_plp}. For the readers convenience we include the proof. Let $\eps_{n}\to0$ and, by abuse of notation, set $\vp^{n}:=\vp^{\ve_{n}}$. Let $u^{n}\in L^{m+1}\cap H^{-1}$ such that $u^{n}\rightharpoonup u$ weakly in $H^{-1}$ for some $u\in H^{-1}$. W.l.o.g.
\[
\liminf_{n\to\infty}\vp^{n}(u^{n})<+\infty
\]
and for a non-relabeled subsequence $\vp^{n}(u^{n})<\infty$ and
\[
\lim_{n\to\infty}\vp^{n}(u^{n})=\liminf_{n\to\infty}\vp^{n}(u^{n})<+\infty.
\]
Hence,
\[
\rho\int_{\Ocal}|u^{n}(\xi)|^{m+1}\, d\xi\le\int_{\Ocal}a\left(\frac{\xi}{\ve_{n}}\right)|u^{n}(\xi)|^{m+1}\, d\xi\le C
\]
which implies that there is a subsequence (again denoted by $u^{n}$) that converges weakly to some $u^{0}\in L^{m+1}(\Ocal)$. By the $H^{-1}$ weak convergence $u^{n}\rightharpoonup u$ we have that $u^{0}=u$. In particular, we conclude that $u\in L^{m+1}\cap H^{-1}$ with 
\[
\int_{\Ocal}|u|^{m+1}\, d\xi\le C.
\]
By Young's inequality, for $\eta\in L^{\frac{m+1}{m}}(\Ocal)$, 
\[
\begin{split}\int_{\Ocal}u^{n}\eta\, d\xi & =\int_{\Ocal}\frac{\eta}{a\left(\frac{\xi}{\ve_{n}}\right)}u^{n}a\left(\frac{\xi}{\ve_{n}}\right)\, d\xi\\
 & \le\vp^{n}(u^{n})+\frac{m}{m+1}\int_{\Ocal}\left|\frac{\eta}{a\left(\frac{\xi}{\ve_{n}}\right)}\right|^{\frac{m+1}{m}}a\left(\frac{\xi}{\ve_{n}}\right)\, d\xi\\
 & =\vp^{n}(u^{n})+\frac{m}{m+1}\int_{\Ocal}|\eta|^{\frac{m+1}{m}}a^{-\frac{1}{m}}\left(\frac{\xi}{\ve_{n}}\right)\, d\xi.
\end{split}
\]
Passing on to the limit, by \cite[Theorem 2.6]{CD99},
\[
\int_{\Ocal}u\eta\, d\xi\le\liminf_{n\to\infty}\vp^{n}(u^{n})+\frac{m}{m+1}\int_{\Ocal}M_{Y}(a^{-\frac{1}{m}})|\eta|^{\frac{m+1}{m}}\, d\xi.
\]
Note that by Jensen's inequality, $M_{Y}(a^{-\frac{1}{m}})\le M_{Y}(a)^{-\frac{1}{m}}$. Hence, setting $\eta:=M_{Y}(a)u^{[m]}\in L^{\frac{m+1}{m}}$, yields,
\[
M_{Y}(a)\int_{\Ocal}|u|^{m+1}\, d\xi\le\liminf_{n\to\infty}\vp^{n}(u^{n})+\frac{m}{m+1}\int_{\Ocal}M_{Y}(a)|u|^{m+1}\, d\xi,
\]
and hence
\[
\vp^{\operatorname{hom}}(u)=\frac{M_{Y}(a)}{m+1}\int_{\Ocal}|u|^{m+1}\, d\xi\le\liminf_{n\to\infty}\vp^{n}(u^{n})
\]
and the first Mosco condition is proved. By \cite[Theorem 2.6]{CD99}, it is easy to see, that for all $u\in L^{m+1}$,
\[
\lim_{n\to\infty}\frac{{1}}{m+1}\int_{\Ocal}a\left(\frac{\xi}{\ve_{n}}\right)|u|^{m+1}\, d\xi=\frac{M_{Y}(a)}{m+1}\int_{\Ocal}|u|^{m+1}\, d\xi.
\]
Hence,
\[
\limsup_{n\to\infty}\vp^{n}(u)\le\vp^{\operatorname{hom}}(u),
\]
for all $u\in H^{-1}$.
\end{proof}

\appendix

\section{Moreau-Yosida approximation of singular powers\label{sec:Moreau-Yosida}}

Let $\psi(\xi):=\frac{1}{p}|\xi|^{p}$, $p\in[1,2)$ and $\phi:=\partial\psi$. We choose $\psi^{\d}$ to be the Moreau-Yosida approximation of $\psi$ (cf.~\cite[p. 266]{A84}). Then $\phi^{\d}:=\partial\psi^{\d}$ is the Yosida approximation of $\phi$, i.e.
\[
\phi^{\d}(\xi)=\frac{1}{\d}(\xi-R_{\d}\xi)\in\phi(R{}_{\d}\xi)\quad\forall\xi\in\R^{d},
\]
where $R_{\d}(\xi)$ is the resolvent of $\phi$, that is, the unique solution $\zeta$ to
\[
\zeta+\delta\phi(\zeta)\ni\xi.
\]
We note that
\begin{equation}
|\phi^{\d}(\xi)|\le|\phi(\xi)|:=\inf\{|\eta|:\eta\in\phi(\xi)\}\quad\forall\xi\in\R^{d}.\label{eq:phi-delta-bound}
\end{equation}
Moreover,
\begin{align}
\psi^{\d}(\xi) & =\frac{1}{2\d}|\xi-R_{\d}\xi|^{2}+\psi(R_{\d}\xi)\label{eq:moreau_yos}\\
 & =\frac{\d}{2}|\phi^{\d}(\xi)|^{2}+\psi(R_{\d}\xi)\quad\forall\xi\in\R^{d}.\nonumber 
\end{align}
Hence,
\begin{align}
\psi(R_{\d}\xi) & \le\psi^{\d}(\xi)\le\psi(\xi)\quad\forall\xi\in\R^{d}.\label{eq:MY-ineq}
\end{align}
By the subgradient inequality we have
\[
(\eta,R_{\d}\xi-\xi)+\psi(\xi)\le\psi(R_{\d}\xi)
\]
for all $\eta\in\phi(\xi)$. Hence, using the definition of $\phi^{\d}$ 
\begin{align*}
\psi(\xi)-\psi(R_{\d}\xi) & \le-(\eta,R_{\d}\xi-\xi)\\
 & \le|\eta|\d|\phi^{\d}(\xi)|
\end{align*}
for every $\eta\in\phi(\xi)$. Hence, using \eqref{eq:phi-delta-bound} and \eqref{eq:MY-ineq} and noting that $p\in[1,2)$, we obtain
\begin{align}
|\psi(\xi)-\psi^{\d}(\xi)| & \le\d|\phi(\xi)|^{2}\label{eq:yosida_convergence-2}\\
 & \le C\d(1+\psi(\xi))\quad\forall\xi\in\R^{d}.\nonumber 
\end{align}
We note
\begin{align}
(\phi^{\d_{1}}(\xi)-\phi^{\d_{2}}(\z))\cdot(\xi-\z)= & (\phi^{\d_{1}}(\xi)-\phi^{\d_{2}}(\z))\cdot(R_{\d_{1}}\xi-R_{\d_{2}}\z)\nonumber \\
 & +(\phi^{\d_{1}}(\xi)-\phi^{\d_{2}}(\z))\cdot(\xi-R_{\d_{1}}\xi-(\z-R_{\d_{2}}\z))\label{eq:phi-bound-1-1}\\
\ge & (\phi^{\d_{1}}(\xi)-\phi^{\d_{2}}(\z))\cdot(\d_{1}\phi^{\d_{1}}(\xi)-\d_{2}\phi^{\d_{2}}(\z))\nonumber \\
\ge & -2(\d_{1}+\d_{2})\left(|\phi^{\d_{1}}(\xi)|^{2}+|\phi^{\d_{2}}(\z)|^{2}\right),\nonumber 
\end{align}
for all $\xi,\z\in\R^{d}$. Since
\[
|\phi^{\d_{1}}(\xi)|^{2}\le|\phi(\xi)|^{2}\le C(1+|\xi|^{2})
\]
we have that
\begin{align}
(\phi^{\d_{1}}(\xi)-\phi^{\d_{2}}(\z))\cdot(\xi-\z) & \ge-C(\d_{1}+\d_{2})(1+|\xi|^{2}+|\z|^{2}),\label{eq:monotone_Y_bound}
\end{align}
for all $\xi,\z\in\R^{d}$.

\section{Mosco convergence of integral functionals\label{sec:Random-Mosco-convergence}}

Let $H$ be a separable Hilbert space and $\varphi:H\to[0,+\infty]$ be a proper, l.s.c., convex functional. By \cite[Theorem 2.8]{B10} the subdifferential $\partial\vp$ is a maximal monotone operator on $H$. For $\lambda>0$, $x\in H$ we define the resolvent $R_{\lambda}^{\partial\varphi}(x)$ as the unique solution $y$ to
\[
y+\lambda\partial\varphi(y)\ni x.
\]
In the following let $\vp^{n}$ be a sequence of proper, l.s.c., convex functionals. 
\begin{defn}
\label{def:Mosco-convergence}We say that $\vp^{n}\to\varphi$ \emph{in Mosco sense} as $n\to\infty$ if
\begin{enumerate}
\item For every sequence  $u^{n}\in H$ such that $u^{n}\rightharpoonup u$ weakly for some $u\in H$ it holds that
\[
\liminf_{n\to\infty}\vp^{n}(u^{n})\ge\varphi(u),
\]

\item For every $v\in H$ there exists a sequence $v^{n}\in H$ such that $v^{n}\to v$ strongly and
\[
\limsup_{n\to\infty}\vp^{n}(v^{n})\le\varphi(v).
\]

\end{enumerate}
\end{defn}

\begin{defn}
~
\begin{enumerate}
\item We say that $\partial\vp^{n}\to\partial\varphi$ \emph{in the strong resolvent sense} if for each $x\in H$, $\lambda>0$ the resolvents converge, i.e. 
\[
R_{\lambda}^{\partial\vp^{n}}(x)\to R_{\lambda}^{\partial\varphi}(x)\quad\text{for }n\to\infty.
\]

\item We say that condition \emph{(N)} holds if there exists a sequence $(u^{n},v^{n})\in H\times H$ and an $(u,v)\in H\times H$ such that $v^{n}\in\partial\vp^{n}(u^{n})$ for all $n\in\N$ and $v\in\partial\varphi(u)$ with $u^{n}\to u$ and $v^{n}\to v$.
\end{enumerate}
\end{defn}
If $0\in\partial\vp^{n}(0)$ for all $n\in\N$ and $0\in\partial\varphi(0)$, then condition \emph{(N)} is trivially satisfied. From \cite[Theorem 3.26]{A84} we recall
\begin{thm}
\label{thm:Mosco-Theorem}We have $\vp^{n}\to\varphi$ in Mosco sense if and only if $\partial\vp^{n}\to\partial\varphi$ in the strong resolvent sense and condition \emph{(N)} holds.
\end{thm}
Let $(\Omega,\mathcal{{A}},\mu)$ be a complete, totally $\sigma$-finite measure space and for $u\in L^{2}(\Omega,\mu;H)$ let 
\begin{align*}
\bar{\varphi}(u): & =\int_{\Omega}\varphi(u(\omega))\,\mu(d\omega)\\
\bar{\varphi}^{n}(u): & =\int_{\Omega}\vp^{n}(u(\omega))\,\mu(d\omega).
\end{align*}
Note that $\bar{\varphi},\bar{\varphi}^{n}$ define convex, l.s.c., proper functionals on $L^{2}(\Omega,\mu;H)$.
\begin{thm}
\label{thm:random-Mosco}Suppose either that condition \emph{(N)} holds for $\bar{\varphi}^{n}$, $\bar{\varphi}$ or that $\mu$ is a finite measure. Then $\vp^{n}\to\varphi$ in Mosco sense implies that $\bar{\varphi}^{n}\to\bar{\varphi}$ in Mosco sense in $L^{2}(\Omega,\mu;H)$.\end{thm}
\begin{proof}
\noindent We follow similar ideas as in \cite{Att78}.

\noindent \textit{Step 1:} By \cite[Theorem 21]{R74} the subdifferential of $ $$\bar{\varphi}$ is given by
\[
\partial\bar{\varphi}(\bar{x}):=\{\bar{v}\in L^{2}(\Omega,\mu;H):\,\bar{v}(\omega)\in\partial\varphi(\bar{x}(\omega))\quad\text{for\,}\mu\text{-a.a.\,}\o\in\O\}.
\]
Let $\bar{x}\in L^{2}(\Omega,\mu;H)$, $\lambda>0$. By definition, the resolvent $R_{\lambda}^{\partial\bar{\varphi}}(\bar{x})$ of $\partial\bar{\varphi}$ is the unique solution $\bar{y}\in L^{2}(\Omega,\mu;H)$ of 
\[
\bar{y}+\lambda\partial\bar{\varphi}(\bar{y})\ni\bar{x}.
\]
Due to the characterization of $\partial\bar{\varphi}$ above this is equivalent to 
\[
\bar{y}(\o)+\lambda\partial\varphi(\bar{y}(\o))\ni\bar{x}(\o)\quad\text{for\,}\mu\text{-a.a.\,}\o\in\O,
\]
i.e. 
\[
\bar{y}(\o)=R_{\lambda}^{\partial\varphi}(\bar{x}(\o))\quad\text{for\,}\mu\text{-a.a.\,}\o\in\O.
\]
Hence,
\[
\left(R_{\lambda}^{\partial\bar{\varphi}}(\bar{x})\right)(\omega)=R_{\lambda}^{\partial\varphi}(\bar{x}(\omega))\quad\text{{for\;}}\mu\text{{-a.a.\;}}\omega\in\Omega.
\]

\noindent \textit{Step 2: }By Theorem \ref{thm:Mosco-Theorem}, for all $x\in H$, $\lambda>0$ we have that 
\[
R_{\lambda}^{\partial\vp^{n}}(x)\to R_{\lambda}^{\partial\varphi}(x)\quad\text{for }n\to\infty
\]
and condition \emph{(N)} holds for $\vp^{n}$, $\varphi$. If $\mu$ is a finite measure, condition \emph{(N)} for $\vp^{n}$, $\varphi$ implies condition \emph{(N)} for $\bar{\varphi}^{n}$, $\bar{\varphi}$. Otherwise it holds by assumption. Using step one we observe that 
\begin{align*}
\left(R_{\lambda}^{\partial\bar{\varphi}^{n}}(\bar{x})\right)(\omega) & =R_{\lambda}^{\partial\vp^{n}}(\bar{x}(\omega))\\
 & \to R_{\lambda}^{\partial\varphi}(\bar{x}(\omega))\\
 & =\left(R_{\lambda}^{\partial\bar{\varphi}}(\bar{x})\right)(\omega)\quad{\text{{for}}\;}\mu\text{{-a.a.\;}}\omega\in\Omega,
\end{align*}
for $n\to\infty.$ By the contraction property of the resolvent (that is $\|R_{\lambda}^{\partial\vp^{n}}(x)\|_{H}\le\|x\|_{H}$ for all $x\in H)$ and by Lebesgue's dominated convergence theorem we conclude

\noindent 
\[
R_{\lambda}^{\partial\bar{\varphi}^{n}}(\bar{x})\to R_{\lambda}^{\partial\bar{\varphi}}(\bar{x})\quad\text{in }L^{2}(\Omega,\mu;H),
\]
for $n\to\infty.$ Applying Theorem \ref{thm:Mosco-Theorem} again, we get the desired convergence $\bar{\varphi}^{n}\to\bar{\varphi}$ in Mosco sense in $L^{2}(\Omega,\mu;H)$ as $n\to\infty$.
\end{proof}

\def\cprime{$'$}

\end{document}